\documentclass{amsart}

\usepackage{amssymb,latexsym,amsmath,graphics,setspace,amscd,verbatim}
\usepackage{epsfig,epsf,amsthm,amsfonts}
\usepackage[active]{srcltx}

\theoremstyle{plain}
\newtheorem{theo}{Theorem}[section]
\newtheorem{prop}[theo]{Proposition}
\newtheorem{lemma}[theo]{Lemma}
\newtheorem{cor}[theo]{Corollary}

\theoremstyle{definition}
\newtheorem{defi}[theo]{Definition}
\newtheorem{remark}[theo]{Remark}

\newcommand{\R}{\mathbb{R}}
\newcommand{\Z}{\mathbb{Z}}

\newcommand{\C}{\mathbb{C}}

\newcommand{\D}{\mathbb{D}}
\newcommand{\E}{\mathcal{E}}

\newcommand{\J}{\mathcal{J}}

\newcommand{\M}{\mathcal{M}}
\newcommand{\W}{\mathcal{W}}

\renewcommand{\P}{\mathcal{P}}
\renewcommand{\sp}{\text{Sp}}

\newcommand{\Ss}{\mathcal{S}}

\newcommand{\U}{\mathcal{U}}
\newcommand{\V}{\mathcal{V}}
\newcommand{\OO}{\mathcal{O}}
\newcommand{\DD}{\mathcal{D}}

\newcommand{\Y}{\mathcal{Y}}
\newcommand{\interior}[1]{\mathring{{#1}}}

\newcommand{\norma}[1]{\left|{#1}\right|}

\newcommand{\est}{e^{2\pi(s+it)}}
\newcommand{\util}{\tilde{u}}
\newcommand{\wtil}{\tilde{w}}
\newcommand{\vtil}{\tilde{v}}
\newcommand{\jtil}{\tilde J}

\newcommand{\wind}{\text{wind}}
\newcommand{\cl}[1]{\overline{{#1}}}
\newcommand{\maslov}{\text{Maslov}}

\newcommand{\sign}{\text{sign}}
\renewcommand{\sl}{\text{sl}}

\begin{document}

\title[On the existence of disk-like global sections]{On the existence of disk-like global sections for Reeb flows on the tight $3$-sphere}
\author{Umberto Hryniewicz}
\address[Umberto Hryniewicz]{Departamento de Matem\'atica Aplicada, Instituto de Matem\'atica - Universidade Federal do Rio de Janeiro}
\email[Umberto Hryniewicz]{umberto@labma.ufrj.br}
\author{Pedro A. S. Salom\~ao}
\address[Pedro A. S. Salom\~ao]{Departamento de Matem\'atica, Instituto de Matem\'atica e Estat\'istica - Universidade de S\~ao Paulo}
\email[Pedro A. S. Salom\~ao]{psalomao@ime.usp.br}
\date{June 20, 2011}

\begin{abstract}
We consider Reeb dynamics on the $3$-sphere associated to a tight contact form. Our main result gives necessary and sufficient conditions for a periodic Reeb orbit to bound a disk-like global section for the Reeb flow, when the contact form is assumed to be non-degenerate.
\end{abstract}

\maketitle

\tableofcontents

\section{Introduction}

In this work we study global dynamical properties of Reeb flows associated to tight contact forms on $S^3$. Recall that a $1$-form $\lambda$ on a $3$-manifold is a contact form if $\lambda \wedge d\lambda$ never vanishes. The Reeb vector field $R$ is uniquely determined by the equations
\begin{equation}\label{reeb_vector}
 \begin{array}{cc}
   i_R d\lambda = 0, & i_R \lambda = 1
 \end{array}
\end{equation}
and its flow $\{\phi_t\}$ is called the Reeb flow. The associated contact structure is the $2$-plane distribution
\begin{equation}\label{contact_str}
 \xi = \ker \lambda.
\end{equation}
An overtwisted disk is an embedded disk $D$ satisfying $T\partial D \subset \xi$, $T_pD \not= \xi_p \ \forall p\in\partial D$, and the contact structure is called tight if there are no overtwisted disks. We abuse the terminology and say that $\lambda$ is tight in this case. As an example, the form $\lambda_0 = \frac{1}{2} \sum_{j=1}^2 q_jdp_j - p_jdq_j$ on $\R^4$ with coordinates $(q_1,p_1,q_2,p_2)$ restricts to a contact form on $S^3$, and a result of Bennequin states that $\lambda_0|_{S^3}$ is tight. If $f:S^3 \to \R\setminus\{0\}$ is smooth then the same is obviously true for $f\lambda_0|_{S^3}$ and, by a deep theorem of Eliashberg, these are precisely the tight contact forms on $S^3$, up to diffeomorphism. The Reeb flow associated to $f\lambda_0|_{S^3}$ is equivalent to the Hamiltonian flow on a star-shaped energy level inside $\R^4$ equipped with its canonical symplectic structure $d\lambda_0$. We look for disks that are global sections in the following sense.

\begin{defi}
Let $\lambda$ be a contact form on $ S^3$.
A disk-like global surface of section for the Reeb dynamics associated to $\lambda$ is an embedded disk $D\hookrightarrow  S^3$ such that $\partial D$ is a closed Reeb orbit, $\interior D$ is transversal to the Reeb vector field and all Reeb trajectories in $S^3 \setminus \partial D$ hit $\interior D$ infinitely often forward and backward in time.
\end{defi}

It is our purpose to investigate the following \\

\noindent {\bf Question:} Given a tight contact form on $S^3$, which closed Reeb orbits are the boundary of some disk-like global surface of section for the associated Reeb flow? \\

If a Reeb orbit is the boundary of a disk which is a global section then it must be unknotted and linked to all other closed orbits. Also, it must satisfy certain contact-topological restrictions since it bounds a disk transversal to the Reeb vector field. This last piece of information is encoded in the self-linking number. Moreover, in the non-degenerate case, the flow must ``twist'' enough with respect to the disk, forcing certain restrictions on the linearized dynamics along the Reeb orbit, which are encoded in its Conley-Zehnder index. For the precise statement of these necessary conditions see Theorem~\ref{main1} below.

\subsection{Main result and sketch of the proof}

We need to consider systems of disk-like global sections as follows.

\begin{defi}
An open book decomposition\footnote{Recall that an open book decomposition of a $3$-manifold $M$ is a pair $(L,p)$ where $L \subset M$ is a link, and $p:M\setminus L \rightarrow S^1$ is a fibration such that each fiber $p^{-1}(\theta)$ is the interior of a compact embedded surface $S_\theta \hookrightarrow M$ satisfying $\partial S_\theta = L$. $L$ is called the binding and the fibers are called pages. If the pages are disks we say that $(L,p)$ has disk-like pages.} with disk-like pages of $S^3$ is said to be adapted to $\lambda$ if its binding consists of a closed Reeb orbit, and each page is a disk-like global surface of section for the Reeb dynamics.
\end{defi}

Our main result is

\begin{theo}\label{main1}
Let $\lambda$ be a non-degenerate tight contact form on $S^3$ and $\bar P=(\bar x,\bar T)$ be a simply covered closed Reeb orbit. Then $\bar P$ bounds a disk-like global section for the Reeb flow if, and only if, $\bar P$ is unknotted, its Conley-Zehnder index $\mu_{CZ}(\bar P)$ is greater than or equal to $3$, $\bar P$ has self-linking number $-1$ and all periodic orbits $P$ satisfying $\mu_{CZ}(P) = 2$ are linked to~$\bar P$. In this case, there exists an open book decomposition with disk-like pages of $S^3$ adapted to $\lambda$ with binding $\bar x(\R)$.
\end{theo}

In the statement above a Reeb orbit $P$ is a pair $(x,T)$, where $x$ is a Reeb trajectory and $T>0$ is a period of $x$. It is called simply covered if $T$ is its prime period, and unknotted when $x(\R)$ is the unknot. We say that $P$ is linked to the unknotted orbit $\bar P = (\bar x,\bar T)$ if the homology class of $t\in \R/\Z\mapsto x(Tt)$ in $H_1(S^3\setminus \bar x(\R),\Z)$ is non-zero. The Conley-Zehnder index, the self-linking number and the non-degeneracy of $\lambda$ are discussed below.

The proof of sufficiency of the above conditions on $\bar P$ is based on the theory of pseudo-holomorphic curves in symplectizations, as introduced by Hofer in his ground-breaking work~\cite{93}. There are two main steps. First, we use disk-filling methods as in~\cite{filling,gromov,93,char1,char2} to obtain finite-energy planes asymptotic to $\bar P$ from a Bishop family of disks with boundary on a suitable disk $\mathcal D$ spanning $\bar P$. The necessary compactness for this analysis comes from the linking assumptions between $\bar P$ and the orbits with index $2$, we explain. If the Bishop family breaks before producing a plane we get a punctured pseudo-holomorphic disk with boundary on $\mathcal D$, all its punctures are negative and asymptotic to Reeb orbits with Conley-Zehnder index $\geq 2$. If one orbit has index $\geq 3$ then the Fredholm index at the punctured disk becomes negative and, consequently, all asymptotic orbits have index $2$ when the almost complex structure is generic. By standard intersection arguments the obtained orbits can not be linked to $\bar P$. This shows that the Bishop family does not brake before it approaches $\partial \mathcal D$. In a second step we prove that the Bishop family can be forced to produce planes with very ``fast'' exponential decay. This decay implies good embeddedness and regularity properties, see~\cite{hryn1}. Then, the linking assumptions on $\bar P$ again provide the necessary compactness for these planes to be pages of an open book decomposition adapted to $\lambda$.

\begin{figure}
  \includegraphics[width=250\unitlength]{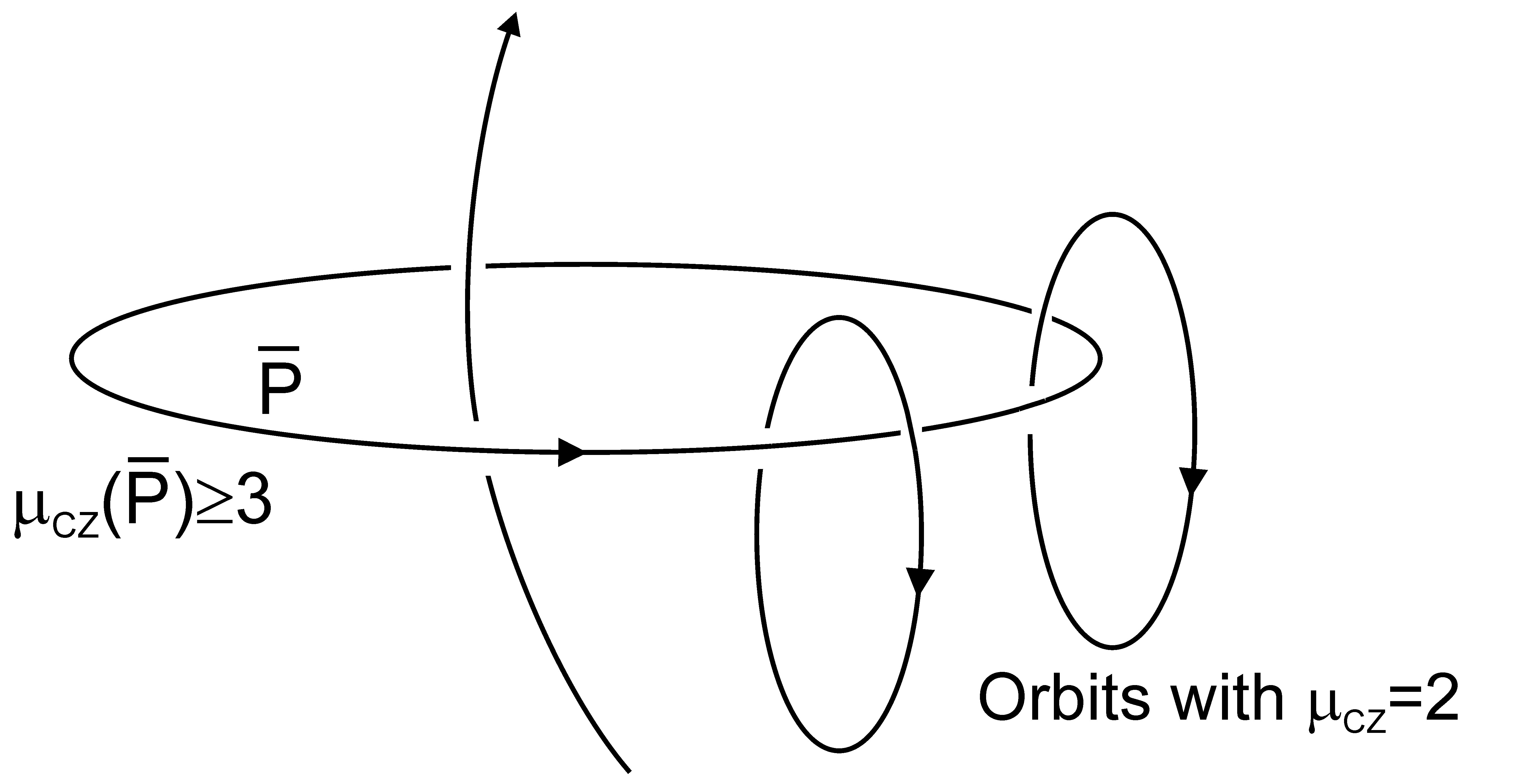}
  \caption{Unknotted orbit $\bar P$ as in Theorem~\ref{main1}.}
\end{figure}

\subsection{Discussion and historical remarks}\label{historical}

The existence question of global sections for Reeb flows was first studied using pseudo-holomorphic curve theory in~\cite{char1,char2,convex} by Hofer, Wysocki and Zehnder. We establish some notation and recall the notions of Conley-Zehnder index and self-linking number before discussing their results.

If $\lambda$ is a tight contact form on $S^3$ and $x$ is a Reeb trajectory of period $T>0$ then
\[
  \begin{array}{cc}
    x_T :  S^1 \simeq \R/\Z \to  S^3, & x_T(t) := x(Tt)
  \end{array}
\]
defines an element of $C^\infty(S^1,S^3)$. We can identify $P=(x,T)$ with the element of $C^\infty(S^1,S^3)/S^1$ induced by the loop $x_T$. Since $\phi_t$ preserves $\lambda$, we have $d\lambda$-symplectic maps $d\phi_t:\xi_{x(0)} \to \xi_{x(t)}$ and $P$ is called non-degenerate if $1$ is not an eigenvalue of $d\phi_{T}|_{\xi_{x(0)}}$. If $T_{min}>0$ is the minimal (or prime) period of $x$ then its multiplicity is $T/T_{min} \in \Z^+$, and $P$ is simply covered if the multiplicity is $1$. The set of periodic Reeb orbits will be denoted by $\P(\lambda)$, or simply by $\P$. If every $P\in\P(\lambda)$ is non-degenerate then $\lambda$ is called non-degenerate. This is a $C^\infty$-generic condition on $\lambda$.

Denoting by $\sp(n)$ the symplectic group in dimension $2n$, consider the set $\Sigma^*(1)$ of paths $\varphi \in C^\infty([0,1],\sp(1))$ such that $\varphi(0) = I$ and $\det (\varphi(1)-I) \not= 0$.
%
%
%
The following statement describes the Conley-Zehnder index axiomatically, see~\cite{fols}.

\begin{theo}\label{cz_index_axioms}
There is a unique surjective map $\mu : \Sigma^*(1) \to \Z$, called the Conley-Zehnder index, satisfying the following axioms:
\begin{enumerate}
 \item The map $s \mapsto \mu(\varphi^s)$ is constant if $\{\varphi^s\}$ is a homotopy of paths in $\Sigma^*(1)$.
 \item If $\psi:[0,1] \to \sp(1)$ is a smooth closed loop based at $I$ then $$ \mu(\psi\varphi) = 2\maslov(\psi) + \mu(\varphi) \ \forall \varphi \in \Sigma^*(1). $$
 \item If $\varphi \in \Sigma^*(1)$ then $\mu(\varphi^{-1}) = -\mu(\varphi)$.
 \item If $\varphi(t) = e^{i\pi t}$ then $\mu(\varphi) = 1$.
\end{enumerate}
\end{theo}

The bilinear form $d\lambda$ turns $\xi=\ker\lambda$ into a trivial symplectic bundle. Fix $P =(x,T) \in \P$. Any global symplectic trivialization of $(\xi,d\lambda)$ represents $d\phi_{Tt}:\xi_{x(0)} \to \xi_{x(Tt)}$, $t\in[0,1]$, as a smooth map $\varphi:[0,1]\to\sp(1)$. If $P$ is non-degenerate then $\varphi \in \Sigma^*(1)$ and we can use the index $\mu$ from Theorem~\ref{cz_index_axioms} to define $$ \mu_{CZ}(P) := \mu(\varphi) \in \Z. $$ This is independent of the chosen global trivialization. The following important definition is originally found in~\cite{convex}.

\begin{defi}[Hofer, Wysocki and Zehnder]
The contact form $\lambda$ on $S^3$ is called dynamically convex if $\mu_{CZ}(P) \geq 3$ for every closed Reeb orbit $P \in \P(\lambda)$.
\end{defi}

The self-linking number $\sl(L)$ of any knot $L \subset S^3$ transverse to $\xi$ is defined as follows. Choose a Seifert surface\footnote{A Seifert surface for $L$ is an orientable embedded connected compact surface $\Sigma$ such that $L = \partial \Sigma$.} $\Sigma$ for $L$ and a smooth non-vanishing section $Z$ of $\xi \to S^3$. $Z$ can be used to slightly perturb $L$ away from itself to another transverse knot $L_\epsilon = \{ \exp_x (\epsilon Z_x) : x\in L \}$. A choice of orientation for $\Sigma$ induces orientations of $L$ and $L_\epsilon$. Then
\begin{equation}\label{defselflink0}
 \sl(L) := L_\epsilon \cdot \Sigma \in \Z,
\end{equation}
where $S^3$ is oriented by $\lambda\wedge d\lambda$ and $L_\epsilon \cdot \Sigma$ is the oriented intersection number of $L$ and $\Sigma$. It is independent of all choices.

Disk-filling methods were first used in~\cite{char1,char2} to obtain disk-like global sections of dynamically convex Reeb flows. For instance, the following theorem is a consequence of results from~\cite{char2} and~\cite{convex}.

\begin{theo}
If the non-degenerate dynamically convex tight contact form $\lambda$ on $S^3$ admits an unknotted prime closed Reeb orbit $P_0=(x_0,T_0)$ satisfying $\mu_{CZ}(P_0)=3$ and $\sl(P_0)=-1$, then $P_0$ is the binding of an open book decomposition with disk-like pages adapted to $\lambda$.
\end{theo}



In particular, $P_0$ bounds a disk-like global section. It is also shown in~\cite{char2} that the tight $3$-sphere can be dynamically characterized as a contact $3$-manifold admitting a pair $(\lambda,P_0)$ with these properties\footnote{A contact form $\lambda$ on a $3$-manifold $M$ is called dynamically convex if $c_1(\ker\lambda)$ vanishes on $\pi_2(M)$ and all contractible closed Reeb orbits have Conley-Zehnder index at least $3$. See~\cite{char2} for details.}. However, the assumption on the Conley-Zehnder index of $P_0$ is too restrictive and can be dropped, see Theorem~\ref{main_hryn1} below.

Reeb flows on the tight $3$-sphere are equivalent to Hamiltonian flows on star-shaped energy levels in $\R^4$, and one obtains dynamically convex contact forms when this energy level is strictly convex, see~\cite{convex}. The following remarkable existence theorem is the main result of~\cite{convex}.

\begin{theo}[Hofer, Wysocki and Zehnder]\label{main_convex}
Let $\lambda$ be any dynamically convex tight contact form on $S^3$. Then there exists a disk-like global surface of section $D \hookrightarrow S^3$ for the Reeb flow such that the closed Reeb orbit $P = \partial D$ satisfies $\mu_{CZ}(P) = 3$.
\end{theo}

Motivated by the above statements, one might ask what are the necessary and sufficient conditions for a closed Reeb orbit of a dynamically convex contact form on $S^3$ to bound a disk-like global section. In the non-degenerate case this question is answered in~\cite{hryn1}, where the corresponding generalized characterization of the tight $3$-sphere\footnote{A closed connected contact $3$-manifold is the tight $3$-sphere if, and only if, it admits a non-degenerate dynamically convex tight contact form with an unknotted closed Reeb orbit having self-linking number $-1$.} is also proved.

\begin{theo}\label{main_hryn1}
Let $\lambda$ be a non-degenerate dynamically convex tight contact form on $S^3$. A closed Reeb orbit $P$ bounds a disk-like global surface of section if and only if $P$ is unknotted and $\sl(P) = -1$. When $P$ satisfies these conditions it is the binding of an open book decomposition with disk-like pages adapted to $\lambda$.
\end{theo}

The proofs of Theorems~\ref{main_convex} and~\ref{main_hryn1}, as well as of the results from~\cite{char2}, strongly rely on dynamical convexity and fail in the general case addressed here. We overcome this difficulty in two steps. First we note that the linking hypothesis of Theorem~\ref{main1} provides the necessary compactness to apply disk-filling methods. Secondly, we show that the Bishop family can be forced to produce fast finite-energy planes. 

It is conjectured in~\cite{fols} that the following dichotomy holds for tight Reeb flows on $S^3$: there are either $2$ or $\infty$-many closed orbits. In~\cite{convex} Hofer, Wysocki and Zehnder explain how the existence of an open book decomposition with disk-like pages that are global sections can be used to confirm the above dichotomy. In fact, the first return map to a page preserves a finite area form. Brouwer's translation theorem gives a second periodic orbit which is simply linked to the binding. Results of J. Franks~\cite{franks} imply the existence of $\infty$-many distinct orbits in the presence of a third orbit. Let us refer to this argument as the ``method of disk-like global sections''. Hofer, Wysocki and Zehnder also show in~\cite{fols} that if the contact form is generic enough (being non-degenerate is not enough here) then the conjecture holds. Theorem~\ref{main1} tells us, in the non-degenerate case, exactly when the above conjecture can be verified by the method of disk-like global sections. \\

\noindent {\bf Organization of the article.} Section~\ref{basic_defns} consists of preliminaries. In~\ref{periodic_orbits} and~\ref{cz_index} we give a more detailed definition of the Conley-Zehnder index. Asymptotic operators, as considered in~\cite{props1}, are defined in~\ref{asymp_operators}. In~\ref{holcurve_theory} we recall the basics of pseudo-holomorphic curve theory in symplectizations. In section~\ref{boff_analysis} we revisit the compactness arguments from~\cite{fols} and prove Proposition~\ref{compactness_theorem_1} which is an important tool for our bubbling-off analysis. The goal of section~\ref{comp_fast_planes} is to prove the compactness result for families of fast planes, Theorem~\ref{comp_fast}. The linking assumptions made in Theorem~\ref{main1} are crucial both in sections~\ref{boff_analysis} and~\ref{comp_fast_planes}. Section~\ref{non_deg_case} is devoted to the proof of sufficiency in Theorem~\ref{main1}. In~\ref{special_spanning_disk} we use results from~\cite{hryn1} to obtain special boundary conditions for the Bishop family considered in~\ref{the_bishop_family}. The existence of a page of the desired open book is proved in subsection~\ref{existence_special_plane_section}, where the results from section~\ref{boff_analysis} are used. In~\ref{ob_decomp} we explain how Theorem~\ref{comp_fast} and results from subsection~\ref{existence_special_plane_section} combine together to conclude the argument. Necessity is proved in section~\ref{section_necessity}. \\

\noindent {\bf Acknowledgements.} This work has its origin when both authors were at the Courant Institute, and we thank Professor Helmut Hofer for creating an active and fruitful research environment. We thank the referee for suggestions that significantly improved the presentation. We also thank Al Momin for many discussions about our results and possible applications. We were partially supported by FAPESP (2006/03829-2, Projeto Tem\'atico). The second author was partially supported by CNPq (304759/2007-4, 303651/2010-5).

\section{Preliminaries}\label{basic_defns}

From now on $\lambda$ denotes a tight contact form on $S^3$, unless otherwise stated.

\subsection{Conley-Zehnder indices of periodic Reeb orbits}\label{periodic_orbits}

Fix $P =(x,T) \in \P(\lambda)$. As explained in the introduction, the vector bundle $\xi_P:= x_T^*\xi \to S^1$ becomes symplectic with the bilinear form $d\lambda$, and we denote
\[
 \Ss_P = \{\text{homotopy classes of smooth symplectic trivializations of }\xi_P\}.
\]
If we fix a class $\alpha \in \Ss_P$ and a smooth symplectic trivialization $\Psi : \xi_P \to  S^1 \times \R^2$ representing it then there is a bijection $\Ss_P \leftrightarrow \Z$ defined as follows. Let $\beta \in \Ss_P$ be represented by $\Phi$. The homotopy class of the smooth closed loop $\Phi_t\circ\Psi_t^{-1}$ of symplectic matrices depends only on $\alpha$ and $\beta$. Here $\Phi_t$ and $\Psi_t$ denote the restrictions of $\Phi$ and $\Psi$ to $\xi_{x(Tt)}$. The function $\beta \in \Ss_P \mapsto \maslov(\Phi_t\circ\Psi_t^{-1}) \in \Z$ is the required bijection, where $\maslov$ is the usual Maslov index for closed symplectic loops, see~\cite{intro}. The formula
\begin{equation}\label{local_linear_Reeb_flow}
  \varphi(t) = \Psi_t \cdot d\phi_{Tt}|_{x(0)} \cdot \Psi_0^{-1}
\end{equation}
defines a map $\varphi \in C^\infty([0,1],\sp(1))$ satisfying $\varphi(0) = I$, and $P$ is non-degenerate if, and only if, $\varphi \in \Sigma^*(1)$. We define $\mu_{CZ}(P,\alpha) := \mu(\varphi) \in \Z$ where $\mu$ is the index described in Theorem~\ref{cz_index_axioms}. It follows that $\mu_{CZ}(P,\beta) = \mu_{CZ}(P,\alpha) + 2\maslov(\Phi_t \circ \Psi_t^{-1})$.

The symplectic bundle $(\xi,d\lambda)$ over $ S^3$ is trivial, so we can find a global symplectic frame. It trivializes $\xi_P$ and singles out a class
\begin{equation}\label{alpha_P}
 \alpha_P \in \Ss_P
\end{equation}
for every $P\in\P$. We define the Conley-Zehnder index of a closed orbit $P$ by
\begin{equation}
 \mu_{CZ}(P) := \mu_{CZ}(P,\alpha_P).
\end{equation}

\begin{remark}\label{marked_point}
In the following we choose a point in the geometric image $x(\R)$ of each periodic Reeb orbit $P = (x,T)$. It is implicit that $x$ maps $0 \in\R$ to this chosen point.
\end{remark}

A number $T>0$ will be called a period if it is the period of some $P \in \P$. We define
\begin{equation}\label{sigma_1}
 \sigma_1 = \inf \{ T>0 : T \text{ is a period} \}.
\end{equation}
If the contact form $\lambda$ on $ S^3$ is non-degenerate then, for a fixed $C>0$, we can define
\begin{equation}\label{sigma}
 \begin{aligned}
  & \sigma_2(C) = \inf \{ |T^\prime - T| : T^\prime,T \leq C \text { are distinct periods} \} \\
  & 0< \sigma(C) < \min \{ \sigma_1, \sigma_2(C) \}.
 \end{aligned}
\end{equation}
These numbers will be of crucial importance for the analysis that follows.

\subsection{A geometric description of the Conley-Zehnder index}\label{cz_index}

Later on we will need a more geometric description of the index. Fix $\varphi \in \Sigma^*(1)$ and define
\begin{equation}
  \begin{array}{ccc}
    \Delta : \C \setminus \{0\} \to \R & \text{by} & \Delta(z) = \frac{\theta(1)-\theta(0)}{2\pi}
  \end{array}
\end{equation}
where $\theta:[0,1]\to\R$ is a continuous function satisfying $\varphi(t)z \in \R^+ e^{i\theta(t)}$. It is shown in~\cite{fols} that $I(\varphi) = \{\Delta(z) : z \in \C \setminus\{0\}\}$ is an interval of length strictly smaller than $1/2$. Since $1$ is not an eigenvalue of $\varphi(1)$ we have that $\partial I(\varphi) \cap \Z = \emptyset$. Thus either this interval contains precisely one integer in its interior or it lies strictly between two consecutive integers. Define
\begin{equation}\label{cz_index_geometric}
  \mu(\varphi) = \left\{ \begin{aligned} & 2k+1 \text{ if } I(\varphi) \subset (k,k+1), \\ & 2k \text{ if } k\in I(\varphi). \end{aligned} \right.
\end{equation}
One can verify that $\mu(\varphi)$ satisfies the axioms of Theorem~\ref{cz_index_axioms} and, consequently, agrees with the Conley-Zehnder index. See the appendix of~\cite{fols}.

The proof of the following lemma is a straightforward consequence of the above description together with an analysis of the spectrum of the linearized Poincar\'e map, and will be omitted. See~\cite{fols} for more details.

\begin{lemma}\label{cz_index_iterates}
Suppose $\lambda$ is non-degenerate and let $P = (x,T)$ be a closed Reeb orbit. For $j\in\Z^+$ consider the orbit $P^j = (x,jT)$, and fix positive integers $l\leq k$. The following assertions hold.
\begin{enumerate}
  \item[(i)] If $\mu_{CZ}(P^k) = 1$ then $\mu_{CZ}(P^l) = 1$.
  \item[(ii)] If $\mu_{CZ}(P^k) \leq 0$ then $\mu_{CZ}(P^l) \leq 0$.
  \item[(iii)] If $\mu_{CZ}(P^k) = 2$ then $k,l$ and $\mu_{CZ}(P^l)$ belong to $\{1,2\}$ and $P$ is hyperbolic. If $l=1$ and $k=2$ then $\mu_{CZ}(P)=1$. 
\end{enumerate}
\end{lemma}

\subsection{Asymptotic operators}\label{asymp_operators}

A smooth complex structure $J$ on $\xi$ is said to be $d\lambda$-compatible if the bilinear form $d\lambda(\cdot,J\cdot)$ is a positive inner-product on $\xi_p$ for every point $p\in S^3$. As is well-known, the set $\J(\xi,d\lambda)$ of such complex structures is non-empty and contractible.

Fix $J \in \J(\xi,d\lambda)$ and $P=(x,T)$. Then $J$ induces the complex structure $x_T^*J$ on $\xi_P$, still denoted by $J$ for simplicity. We have an unbounded closed operator
\begin{equation}\label{asymp_op}
 \begin{aligned}
  A_P &: W^{1,2}(\xi_P) \subset L^2(\xi_P) \to L^2(\xi_P) \\
  \eta &\mapsto J(-\nabla_t\eta + T\nabla_\eta R).
 \end{aligned}
\end{equation}
Here $R$ is the Reeb vector and $\nabla$ is any symmetric connection on $T S^3$. The reader easily checks that $A_P$ is independent of $\nabla$.

\begin{remark}[Winding Numbers]\label{winding_numbers}
Let $(E,J)$ be any complex line bundle over $S^1$, and consider two non-vanishing sections $Z$ and $W$ of $E$. We find smooth functions $a,b: S^1\to\R$ satisfying $W = aZ + bJZ$ since $\{Z,JZ\}$ is pointwise linearly independent. The function $f: S^1\to\C$ defined by $f(t) = a(t) + ib(t)$ does not vanish and we define
\begin{equation}\label{}
  \wind(W,Z) = \deg \frac{f}{|f|} \in \Z.
\end{equation}
This integer depends only on the homotopy classes of the non-vanishing sections $Z$ and $W$, and on the homotopy class of the fiberwise complex structure $J$. Here we understand $S^1$ as oriented counter-clockwise. If $(E,\Omega)$ is a rank-$2$ symplectic bundle over $S^1$ then there is a preferred homotopy class of complex structures on $E$, namely, the ones which are $\Omega$-compatible. Hence $\wind(W,Z)$ also makes sense for non-vanishing sections $W,Z$ in this case.
\end{remark}

The bundle $\xi_P$ is symplectic with the $2$-form $d\lambda$. Let $W$ be a non-vanishing section $\xi_P$ and choose $\alpha \in \Ss_P$. If $\Psi$ is some trivialization in class $\alpha$ then we can define
\[
 \wind(W,\alpha) = \wind(W,\{t\mapsto \Psi_t^{-1} \cdot e_1\})
\]
where $\Psi_t$ is the restriction of $\Psi$ to $\xi_{x(Tt)}$ and $e_1 = (1,0) \in \R^2$. Here the winding number is computed with respect to any $J \in \J(\xi,d\lambda)$. We shall say that $W$ does not wind with respect to $\alpha$ if $\wind(W,\alpha) = 0$. This is equivalent to saying that there exists a symplectic trivialization $\Psi'$ in class $\alpha$ such that $\Psi'_t\cdot W(t) = e_1 \ \forall t\in S^1$.

Choose a symplectic trivialization $\Psi$ of $\xi_P$. As in subsection~\ref{periodic_orbits} we denote $\varphi(t) = \Psi_t \cdot d\phi_{Tt} \cdot \Psi_0^{-1} \in \sp(1)$ the local representation of the linearized Reeb flow given by $\Psi$. Let $\eta$ be any section of $\xi_P$ and define $v :  S^1 \to \R^2$ by $v(t) = \Psi_t \cdot \eta(t)$. Then
\[
 \Psi_t \cdot (A_P\eta)(t) = -J(t) \dot v (t) + S(t)v(t),
\]
where $J(t) = \Psi_t \cdot J \cdot \Psi_t^{-1}$ and $S(t) = J(t)\dot\varphi(t)\varphi^{-1}(t)$. The reason is that the linear flow generated by the equation
\[
  -\nabla_t\zeta + T\nabla_\zeta R = 0
\]
is precisely $d\phi_{Tt}$, that is, the solution of the above system is the section
\[
 \zeta(t) = d\phi_{Tt} \cdot \zeta_0
\]
where $\zeta(0) = \zeta_0 \in \xi_{x(0)}$. Since $J$ is $d\lambda$-compatible, $\left< \cdot,-J_0J(t)\cdot \right>$ is an inner-product on $\R^2$ and the matrix $S(t)$ is symmetric with respect to this inner-product. A number of non-trivial facts about $A_P$ follows from this local representation, as pointed out in~\cite{props2}. We summarize them in the following statements.

\begin{lemma}[Hofer, Wysocki and Zehnder]\label{winding_spectrum_1}
$A_P$ has a discrete real spectrum accumulating only at $\pm\infty$. The geometric and algebraic multiplicities of each eigenvalue agree and this multiplicity is either $1$ or $2$. Fix eigenvalues $a_j$ of $A_P$ and let $\eta_j$ be a non-trivial eigenfunction of $a_j$, for $j=1,2$. Fix any $\alpha \in \Ss_P$.
\begin{enumerate}
 \item If $a_1=a_2$ then $\wind(\eta_1,\alpha) = \wind(\eta_2,\alpha)$.
 \item If $a_1 \not= a_2$ and $\wind(\eta_1,\alpha) = \wind(\eta_2,\alpha)$ then $\eta_1$ and $\eta_2$ are pointwise linearly independent.
 \item If $a_1\leq a_2$ then $\wind(\eta_1,\alpha) \leq \wind(\eta_2,\alpha)$.
\end{enumerate}
\end{lemma}

Let $a$ be an eigenvalue of $A_P$ and $\alpha \in \Ss_P$. In view of the lemma above we can define
\begin{equation}\label{}
  \wind(a,\alpha) := \wind(\eta,\alpha)
\end{equation}
where $\eta$ is any non-trivial eigenfunction of $a$.

\begin{lemma}[Hofer, Wysocki and Zehnder]\label{winding_spectrum_2}
Fix any $\alpha \in \Ss_P$. For every $k\in\Z$ there are precisely two eigenvalues $a$ and $b$ of $A_P$ (counting multiplicities) such that
\[
 k = \wind(a,\alpha) = \wind(b,\alpha).
\]
\end{lemma}

It was shown in~\cite{props2} that the Conley-Zehnder index can be described in terms of the spectrum of $A_P$ in the following manner. The orbit $P = (x,T)$ is non-degenerate if, and only if, $0$ is not an eigenvalue of $A_P$. In any case, set
\begin{equation}\label{extreme_evalues}
 \begin{aligned}
  & \nu^{neg} = \max \{ \nu < 0 \text{ is an eigenvalue of } A_P \} \\
  & \nu^{pos} = \min \{ \nu > 0 \text{ is an eigenvalue of } A_P \}.
 \end{aligned}
\end{equation}
Fix some $\alpha \in \Ss_P$ and define $p = \frac{1}{2} (1+(-1)^b)$ where $b$ is the number of strictly negative eigenvalues (counted with multiplicities) which have the same winding as $\nu^{neg}$. We have the formula
\begin{equation}\label{cz_index_alternative}
 \mu_{CZ}(P,\alpha) = 2\wind(\nu^{neg},\alpha) + p.
\end{equation}
This provides an extension of the Conley-Zehnder index to degenerate orbits. See~\cite{convex} for more details.

\subsection{Pseudo-holomorphic curves in symplectizations}\label{holcurve_theory}

We now recall the basic definitions and facts from the theory of pseudo-holomor\-phic curves in symplectizations introduced by Hofer and later developed by Hofer, Wysocki and Zehnder. Pseudo-holomorphic curves were introduced in symplectic geometry by M. Gromov in~\cite{gromov}. In his fundamental work~\cite{93} H. Hofer was the first to use them in the non-compact setting of symplectizations in order to study the three dimensional Weinstein conjecture. We now recall the most basic definitions of this theory and refer the reader to the literature as we proceed.

In this subsection $\lambda$ will be a contact form on the closed $3$-manifold $M$ and $\xi = \ker \lambda$ the induced co-oriented contact structure. The Reeb vector field $R$ is defined as before in (\ref{reeb_vector}). There is a projection
\begin{equation}\label{projection}
 \pi : TM \to \xi
\end{equation}
uniquely determined by $\ker \pi = \R R$. As remarked in subsection~\ref{asymp_operators}, the set $\J(\xi,d\lambda)$ of $d\lambda$-compatible complex structures on $\xi$ is non-empty and contractible and, from now on, we fix some $J \in \J(\xi,d\lambda)$. The equations
\begin{equation}\label{jtil}
  \left\{ \begin{aligned} & \jtil \partial_a = R \\ & \jtil|_\xi = J \end{aligned} \right.
\end{equation}
induce an almost complex structure $\jtil$ on $\R \times M$, where $a$ is the $\R$-component in $\R\times M$.

\subsubsection{Finite energy surfaces}

\begin{defi}[Hofer]\label{FES}
Let $(S,j)$ be a closed Riemann surface and $\Gamma \subset S$ be a finite set. A map $\util : S \setminus \Gamma \rightarrow \R\times M$ is called a finite-energy surface if it is $\jtil$-holomorphic, that is, it satisfies the non-linear Cauchy-Riemann equations
\begin{equation}\label{cr}
 \bar\partial_{\jtil}(\util) = \frac{1}{2} (d\util + \jtil(\util) \cdot d\util \cdot j) =0
\end{equation}
and also an energy condition $0<E(\tilde{u})<+\infty$. The energy $E(\tilde{u})$ is defined as follows. Set $\Lambda := \{\phi\in C^\infty(\R,[0,1]):\phi^{\prime}\geq0\}$ and define
\[
 E(\tilde{u})=\sup_{\phi\in\Lambda}\int_{S\setminus\Gamma} \tilde{u}^*d(\phi\lambda).
\]
If $S = S^2$ and $\#\Gamma=1$ we call $\util$ a finite-energy plane.
\end{defi}

The elements of $\Gamma$ are called punctures. Write $\util = (a,u)$, choose $z \in \Gamma$, a holomorphic chart $\varphi : (U,0) \to (\varphi(U),z)$ centered at $z$ and set $\util(s,t) = \util\circ\varphi(e^{-2\pi(s+it)})$ for $s\gg1$. Condition $E(\util)<\infty$ implies that
\begin{equation}\label{mass_defn}
 m = \lim_{s\to+\infty} \int_{\{s\}\times S^1} u^*\lambda
\end{equation}
exists. The puncture $z$ is called removable if $m=0$, positive if $m>0$ and negative if $m<0$. A removable puncture can be removed by an application of Gromov's removable singularity theorem. The $d\lambda$-energy of $\util$ is defined by
\[
 \int_{S\setminus\Gamma} u^*d\lambda.
\]
It is clearly bounded by $E(\util)$ from above. Note that at least one puncture is positive and non-removable.

\begin{theo}[Hofer]\label{thm93}
In the notation explained above, suppose $z$ is non-removable and let $\epsilon=\pm1$ be the sign of $m$ in (\ref{mass_defn}). Then every sequence $s_n \rightarrow +\infty$ has a subsequence $s_{n_k}$ such that the following holds: there exists a real number $c$ and a periodic Reeb orbit $P = (x,T)$ such that $u(s_{n_k},t) \rightarrow x(\epsilon Tt+c)$ in $C^\infty(S^1,M)$ as $k \rightarrow +\infty$.
\end{theo}


\subsubsection{Asymptotic behavior}

\begin{defi}\label{behavior}
Let $(S,j)$, $\Gamma$ and $\util$ be as in Definition~\ref{FES}. Fix a non-removable puncture $z \in \Gamma$, choose a holomorphic chart $\varphi : (U,0) \rightarrow (\varphi(U),z)$ centered at $z$ and write $\util (s,t) = (a(s,t),u(s,t)) = \util \circ \varphi ( e^{-2\pi(s+it)} )$ for $s\gg1$. Define $m$ by~\eqref{mass_defn} and let $\epsilon = \pm1$ be its sign. We say that $z$ is a \textbf{non-degenerate puncture of $\util$} if there exists a periodic Reeb orbit $P = (x,T)$ and constants $c,d\in\R$ such that
\begin{enumerate}
 \item $\sup_{t\in S^1} \norma{a(s,t) - \epsilon Ts - d} \rightarrow 0$ as $s\rightarrow +\infty$.
 \item $u(s,t) \rightarrow x(\epsilon Tt + c)$ in $C^0(S^1,M)$ as $s\rightarrow +\infty$.
 \item If $\pi \cdot du$ does not vanish identically over $S\setminus\Gamma$ then $\pi \cdot du (s,t) \not=0$ when $s$ is large enough.
 \item If we define $\zeta(s,t)$ by $u(s,t) = \exp_{x(\epsilon Tt+c)} \zeta(s,t)$ then $\exists b>0$ such that $\sup_{t\in S^1} e^{bs}\norma{\zeta(s,t)} \rightarrow 0$ as $s\rightarrow +\infty$.
\end{enumerate}
Here $\pi$ is the projection (\ref{projection}). This definition is independent of the holomorphic chart $\varphi$ and the exponential map. In this case we say $\util$ is asymptotic to $P$ at $z$ and that $\util$ has non-degenerate asymptotic behavior at $z$. We say that $\util$ has non-degenerate asymptotics if this holds for every puncture.
\end{defi}

The following theorem is a partial result extracted from~\cite{props1}, where a much more precise asymptotic behavior is described.

\begin{theo}[Hofer, Wysocki and Zehnder]\label{behavior_HWZ_1}
If $\lambda$ is non-degenerate then every finite-energy surface has non-degenerate asymptotics.
\end{theo}

\subsubsection{Algebraic invariants}\label{alg_inv}

Let $\util$, $S$ and $\Gamma$ be as in Definition~\ref{FES} and assume $\util$ has non-degenerate asymptotics. If we write $\util = (a,u)$ then $\pi \cdot du$ is a section of the complex line bundle $\E = \wedge^{0,1}T^*(S\setminus\Gamma) \otimes_\C u^*\xi$ over $S\setminus \Gamma$. It follows from (\ref{cr}) that it satisfies a Cauchy-Riemann type equation
\[
 \pi \cdot du \cdot j = J(u) \cdot \pi \cdot du.
\]
Thus either $\pi \cdot du \equiv 0$ or the zeros are all isolated and count positively for the oriented intersection number with the zero section of $\E$. This follows from an application of Carleman's Similarity Principle, see~\cite{props2}. Consequently $$ \pi \cdot du \equiv 0 \Leftrightarrow \int_{S\setminus\Gamma} u^*d\lambda = 0. $$
When $\pi \cdot du \not\equiv 0$, Hofer, Wysocki and Zehnder define in~\cite{props2}
\begin{equation}\label{wind_pi}
 \wind_\pi(\util) = \# \{ \text{zeros of } \pi \cdot du \}
\end{equation}
where the zeros are counted with multiplicity. In this case, the above remarks show that $0 \leq \wind_\pi(\util) < \infty$
coincides with the oriented intersection number of $\pi \cdot du$ and the zero section of $\E$.

We know that the bundle $u^*\xi$ is trivial as a complex (or symplectic) bundle since $\Gamma \not= \emptyset$ and $u^*\xi$ carries a symplectic form. Now consider a non-vanishing section $Z$ of $u^*\xi$. If $z\in\Gamma$ is non-removable choose a holomorphic chart $\varphi : (B_1(0),0) \to (\varphi(B_1(0)),z)$ for $S$ centered at $z$ and write $u(s,t) = u(\varphi(e^{-2\pi(s+it)}))$ for $s>0$. Define
\[
 \wind_\infty(\util,z,Z) = \lim_{s\to+\infty} \wind(t\mapsto \pi\cdot \partial_su(s,\epsilon t),t\mapsto Z(u(s,\epsilon t))
\]
where $\epsilon=\pm 1$ is the sign of $z$ as a puncture. Here the winding numbers are computed with respect to any $d\lambda$-compatible $J:\xi\to\xi$, as explained after Remark~\ref{winding_numbers}. This is independent of the choice of $\varphi$ and $J$. Now assume $\Gamma$ consists of non-removable punctures and split $\Gamma = \Gamma^+ \sqcup \Gamma^-$ into positive and negative punctures. Finally we define
\begin{equation}\label{}
  \wind_\infty(\util) = \sum_{z\in\Gamma^+} \wind_\infty(\util,z,Z) - \sum_{z\in\Gamma^-} \wind_\infty(\util,z,Z).
\end{equation}
It follows from standard degree theory that this number is independent of $Z$.

\begin{remark}\label{remark_winds_s3}
Later on we will be dealing only with the tight contact structure on $S^3$, which is a trivial bundle. Then we shall choose a global non-vanishing section $Z : S^3 \to \xi$ and write simply $\wind_\infty(\util,z)$ instead of $\wind_\infty(\util,z,Z)$.
\end{remark}

\begin{theo}[Hofer, Wysocki and Zehnder]\label{winds_thm}
Let $\util = (a,u)$ be a finite-energy surface defined on $S\setminus\Gamma$, where $(S,j)$ is a closed Riemann surface and $\Gamma \subset S$ is a finite set consisting of non-removable punctures. If $\util$ has non-degenerate asymptotics then $$ \wind_\pi(\util) = \wind_\infty(\util) - \chi(S) + \#\Gamma. $$
\end{theo}

Since $\wind_\pi(\util)$ is always non-negative, it follows from the above theorem that $\wind_\infty(\util) \geq 1$ for planes.

\begin{defi}[Fast Planes]
A finite-energy plane $\util$ is called fast if it has non-degenerate asymptotics, its asymptotic limit at $\infty$ is a simply covered closed Reeb orbit and $\wind_\pi(\util) = 0$.
\end{defi}

\subsubsection{Surfaces with vanishing $d\lambda$-energy}

The following important theorem is pro\-ved in~\cite{props2}.

\begin{theo}[Hofer, Wysocki and Zehnder]\label{zero_dlmabda_FES}
Let $\util = (a,u) : \C \setminus \Gamma \to \R \times M$ be a finite-energy punctured sphere, where $\Gamma \subset \C$ is the finite set of negative punctures and $\infty$ is the unique positive puncture. If $\pi \cdot du \equiv 0$ then there exists a non-constant polynomial $p: \C \to \C$ and a periodic Reeb orbit $P = (x,T) \in \P$ such that
\[
\begin{array}{ccc}
  p^{-1}(0) = \Gamma & \text{ and } & \util = F_P \circ p
\end{array}
\]
where $F_P : \C \setminus\{0\} \to \R\times M$ is defined by $F_P(z=\est) = (Ts,x(Tt))$.
\end{theo}

\begin{remark}\label{rmk_simp_cov}
In the above statement note that if $p$ has degree $k$ then the asymptotic limit of $\util$ at $\infty$ is $P^k=(x,kT)$.
\end{remark}

\begin{cor}
If $\util=(a,u)$ is a finite-energy plane then $\int_\C u^*d\lambda > 0$.
\end{cor}

\subsubsection{Cylinders with small energy}

The following theorem is Lemma 4.9 from ~\cite{fols}, and is crucial for the so-called ``soft-rescalling''. See also~\cite{long}.

\begin{theo}[Hofer, Wysocki and Zehnder]\label{small_energy}
Fix numbers $C>0$, $e>0$ and let $\W$ be a $ S^1$-invariant open neighborhood of the set of loops
\[
 \{ t \in  S^1 \mapsto x(Tt) \in M : P = (x,T) \in \P \}
\]
in the space $C^\infty( S^1,M)$. Assume that $\lambda$ is non-degenerate and that if $P$ and $\hat P$ are distinct closed Reeb orbits then the loops $t\mapsto x(Tt + c)$ and $t\mapsto \hat x (\hat Tt + d)$ belong to distinct components of $\W$, for every $c,d \in \R$. Let $\sigma(C)$ be defined by (\ref{sigma}). Then there exists $h>0$ such that the following holds. If $$ \util = (a,u) : [r,R] \to \R \times M $$ is a smooth map satisfying
\[
  \bar\partial_{\jtil}(\util) = 0,\ E(\util) \leq C,\ \int_{[r,R]\times S^1} u^*d\lambda \leq \sigma(C) \text{ and } \int_{\{r\}\times  S^1} u^*\lambda \geq e
\]
and if $r+h \leq R-h$ then each loop $t \in  S^1 \mapsto u(s,t) \in M$ with $s\in [r+h,R-h]$ belongs to $\W$.
\end{theo}

\section{Bubbling-off Analysis}\label{boff_analysis}

In this section we make the following important standing assumption:
\[
  \lambda \text{ is a non-degenerate tight contact form on } S^3.
\]
We fix a Riemannian metric on $\R \times  S^3$ of the form
\begin{equation}\label{metric_g0}
 g_0 = da \otimes da + \pi_{ S^3}^*g
\end{equation}
where $g$ is some Riemannian metric on $ S^3$. Here $\pi_{ S^3} : \R \times  S^3 \to  S^3$ and $a : \R \times  S^3 \to \R$ are projections onto the second and first coordinates, respectively. For any point $(r,p) \in \R\times  S^3$ and any real linear map $L:\C \to T_{(r,p)}(\R\times  S^3)$ we shall always denote by $|L|$ the norm induced by the Euclidean inner-product of $\C$ and by the metric $g_0$.

We also fix some $J \in \J(\xi,d\lambda)$ which, in turn, induces the almost complex structure $\jtil$ on $\R\times  S^3$ by formula \eqref{jtil}.

\subsection{Estimates on the Conley-Zehnder indices}

The following theorem can be extracted from~\cite{props1}.

\begin{theo}[Hofer, Wysocki and Zehnder]\label{asymp_efunction}
Let $(S,j)$ be a closed Riemann surface, $\Gamma \subset S$ be a finite set and
\[
 \vtil = (b,v) : S \setminus \Gamma \to \R \times  S^3
\]
be a finite energy surface such that $\Gamma$ consists of non-removable punctures. Fix $z \in \Gamma$ and assume $\vtil$ is asymptotic to $P=(x,T)$ at $z$. Choose a holomorphic chart $\psi:(B_1(0),0) \to (V,z)$ centered at $z$. Write $\vtil(s,t) = \vtil(\psi(e^{-2\pi(s+it)}))$ for $(s,t) \in \R^+ \times  S^1$ if $z$ is a positive puncture, or $\vtil(s,t) = \vtil(\psi(e^{2\pi(s+it)}))$ for $(s,t) \in \R^- \times  S^1$ if $z$ is a negative puncture. By rotating the chart we can assume $v(s,t) \to x(Tt)$ in $C^\infty( S^1, S^3)$ as $|s|\to+\infty$. Then either $\pi \cdot dv$ vanishes identically or
\begin{enumerate}
 \item If $z$ is a positive puncture then there exists a smooth non-vanishing function $f : \R^+\times  S^1 \to \R$ such that
     \[
      \lim_{s\rightarrow+\infty} f(s,t) \pi\cdot\partial_sv = \eta(t)\text{ in } C^\infty(S^1,\xi)
     \]
     where $\eta$ is an eigenfunction of $A_P$ associated to an eigenvalue $\nu < 0$.
 \item If $z$ is a negative puncture then there exists a smooth non-vanishing function $f : \R^-\times  S^1 \to \R$ such that
     \[
      \lim_{s\rightarrow-\infty} f(s,t) \pi\cdot\partial_sv = \eta(t)\text{ in } C^\infty(S^1,\xi)
     \]
     where $\eta$ is an eigenfunction of $A_P$ associated to an eigenvalue $\nu > 0$.
\end{enumerate}
\end{theo}

Note that the above statement is valid if one assumes, as we have, that the contact form is non-degenerate.

\begin{cor}\label{winds_spec}
Let $(S,j)$, $\Gamma$, $\vtil = (b,v)$, $z\in\Gamma$ and $P=(x,T)$ be as in the statement above. Define $\nu^{pos}$ and $\nu^{neg}$ by (\ref{extreme_evalues}). If $\pi \cdot dv$ does not vanish identically then
\begin{enumerate}
 \item $\wind_\infty(\vtil,z) \leq \wind(\nu^{neg},\alpha_P)$ if $z$ is a positive puncture.
 \item $\wind_\infty(\vtil,z) \geq \wind(\nu^{pos},\alpha_P)$ if $z$ is a negative puncture.
\end{enumerate}
\end{cor}

The proof of the above corollary follows immediately from the definition of $\wind_\infty$, from Theorem~\ref{asymp_efunction} and from Lemma~\ref{winding_spectrum_1}. See also Remark~\ref{remark_winds_s3}.

\begin{lemma}\label{wind1}
Let $\Gamma = \{z_1,\dots,z_N\} \subset \C$ be non-empty and finite, and let $$ \util = (a,u) : \dot \C = \C \setminus \Gamma \to \R\times  S^3 $$ be a finite-energy surface with exactly one positive puncture at $\infty$, and negative punctures at the points of $\Gamma$. If $\util$ is asymptotic to $P_\infty$ at $\infty$ and is asymptotic to $P_j$ at $z_j$ ($j=1\dots N$) then the following assertions are true.
\begin{enumerate}
 \item If $\int_{\dot \C} u^*d\lambda > 0$ and $\mu_{CZ}(P_\infty) \leq 1$ then $\mu_{CZ}(P_{j_0}) \leq 1$ for some $j_0 \in \{1,\dots,N\}$.
 \item If $\int_{\dot \C} u^*d\lambda = 0$ and $\mu_{CZ}(P_\infty) = 1$ then $\mu_{CZ}(P_j) = 1$ $\forall j=1,\dots N$.
 \item If $\int_{\dot \C} u^*d\lambda = 0$ and $\mu_{CZ}(P_\infty) \leq 0$ then $\mu_{CZ}(P_j) \leq 0$ $\forall j=1,\dots N$.
\end{enumerate}
\end{lemma}

\begin{proof}
Equations (\ref{extreme_evalues}) define special eigenvalues $\nu^{pos}_j>0>\nu^{neg}_j$ of the asymptotic operators $A_{P_j}$. In the same way we have special eigenvalues $\nu^{pos}_\infty>0>\nu^{neg}_\infty$ of the asymptotic operator $A_{P_\infty}$.

Suppose, by contradiction, that $\mu_{CZ}(P_j) \geq 2$ for every $j=1\dots N$, $\mu_{CZ}(P_\infty) \leq 1$ and that $\pi\cdot dv$ does not vanish identically. It follows easily from this assumption and from (\ref{cz_index_alternative}) that $\wind(\nu^{pos}_j,\alpha_{P_j}) \geq 1$ for every $j=1\dots N$. By the same reasoning we can estimate $\wind(\nu^{neg}_\infty,\alpha_{P_\infty}) \leq 0$.
By Corollary~\ref{winds_spec} we have
\begin{gather*}
 \wind_\infty(\util,\infty) \leq \wind(\nu^{neg}_\infty,\alpha_{P_\infty}) \leq 0 \\
 \wind_\infty(\util,z_j) \geq \wind(\nu^{pos}_j,\alpha_{P_j}) \geq 1 \ \forall j=1\dots N.
\end{gather*}
Theorem~\ref{winds_thm} implies
\[
 \begin{aligned}
  1-N &\leq \wind_\pi(\util) + 1 - N \\
  &= \wind_\pi(\util) + \chi( S^2) - (N+1) = \wind_\infty(\util) \\
  &= \wind_\infty(\util,\infty) - \sum_{j=1}^N \wind_\infty(\util,z_j) \\
  &\leq 0 - \sum_{j=1}^N 1 = -N
 \end{aligned}
\]
which is a contradiction. The first assertion is proved.

Assume $\pi \cdot du$ vanishes identically. If $P_\infty = (x_\infty,T_\infty)$ then denote by $\tau>0$ the minimal positive period of the Reeb trajectory $x_\infty$. Theorem~\ref{zero_dlmabda_FES} implies that each orbit $P_j$ is of the form $P_j = (x_\infty,k_j\tau)$ with integers $k_j\geq 1$, and that $T_\infty = \sum_{j=1}^N k_j\tau$. It follows from Lemma~\ref{cz_index_iterates} that 
\[
 \begin{aligned}
  & \mu_{CZ}(P_\infty) = 1 \Rightarrow \mu_{CZ}(P_j) = 1 \ \forall j=1\dots N, \\
  & \mu_{CZ}(P_\infty) \leq 0 \Rightarrow \mu_{CZ}(P_j) \leq 0 \ \forall j=1\dots N.
 \end{aligned}
\]
\end{proof}

\begin{lemma}\label{wind2}
Let $\Gamma = \{z_1,\dots,z_N\} \not= \emptyset$, $\dot \C$, $\util = (a,u)$, $P_\infty,P_1,\dots,P_N$ be exactly as in the statement of Lemma~\ref{wind1}. Suppose that at least one of the following assertions is true.
\begin{enumerate}
  \item[(i)] $\mu_{CZ}(P_j) \geq 2$ $\forall j=1\dots N$, $\int_{\dot \C} u^*d\lambda > 0$ and $\wind_\infty(\util,\infty) \leq 1$.
  \item[(ii)] $\mu_{CZ}(P_j) \geq 2$ $\forall j=1\dots N$ and $\mu_{CZ}(P_\infty) \leq 2$.
\end{enumerate}
Then $\mu_{CZ}(P_j) = 2$ $\forall j=1\dots N$.
\end{lemma}

\begin{proof}
Assume (i). As in the proof of the previous lemma, we note that it follows from our assumptions, from Theorem~\ref{asymp_efunction}, Corollary~\ref{winds_spec} and (\ref{cz_index_alternative}) that $\wind_\infty(\util,z_j) \geq 1$ for $j=1,\dots,N$. Suppose, by contradiction, that there exists $j_0$ such that $\mu_{CZ}(P_{j_0}) \geq 3$. Then, as before we have $\wind_\infty(\util,z_{j_0}) \geq 2$. Consequently
\[
 \begin{aligned}
  1-N &= \chi( S^2) - (N+1) \\
  &\leq \wind_\pi(\util) + \chi( S^2) - (N+1) \\
  &= \wind_\infty(\util) \\
  &= \wind_\infty(\util,\infty) - \sum_{j=1}^N \wind_\infty(\util,z_j) \\
  &\leq 1 - (N-1) - 2 \\
  &= -N.
 \end{aligned}
\]
This is a contradiction.

Now assume (ii). If $\int_{\dot \C} u^*d\lambda > 0$ then $\mu_{CZ}(P_\infty) \leq 2$ implies $\wind_\infty(\util,\infty) \leq 1$, and we have (i). If $\int_{\dot \C} u^*d\lambda = 0$, $P_\infty = (x_\infty,T_\infty)$ and $\tau>0$ is the minimal positive period of $x_\infty$ then it follows easily from Theorem~\ref{zero_dlmabda_FES} that each orbit $P_j$ is of the form $P_j = (x_\infty,k_j\tau)$ with integers $k_j\geq 1$ satisfying $T_\infty = \sum_{j=1}^N k_j\tau$. Lemma~\ref{cz_index_iterates} implies that $\mu_{CZ}(P_j) \in \{1,2\}$ for $j=1\dots N$. Our assumptions imply $\mu_{CZ}(P_j) = 2$ for $j=1\dots N$.
\end{proof}

\subsection{The bubbling-off tree}\label{bubbling_section}

The following statement is a standard tool for the bubbling-off analysis.

\begin{lemma}\label{standardtool}
Fix any contact form $\lambda$ on $S^3$ and any $J \in \J(\xi=\ker\lambda,d\lambda)$. Let $\Gamma \subset \C$ be finite, $U_k \subset \C\setminus\Gamma$ be an increasing sequence of open sets such that $\cup_k U_k = \C\setminus\Gamma$ and $f_k$ be a sequence of positive smooth real functions on $S^3$ satisfying $f_k\to1$ in $C^\infty$. Define $\lambda_k=f_k\lambda$ and suppose $J_k \in\J(\xi,d\lambda_k)$ satisfy $J_k\to J$ in $C^\infty$. Let $\vtil_k = (b_k,v_k) : U_k \rightarrow \R \times  S^3$ be a sequence of $\jtil_k$-holomorphic maps satisfying $E(\vtil_k) \leq C < \infty \ \forall k$, and $z_k \in U_k$ be a sequence satisfying $|d\vtil_k(z_k)| \rightarrow +\infty$ and $\inf_k \text{dist}(z_k,\Gamma) > 0$. Suppose further that $\{z_k\}$ is bounded or that $\exists n$ such that $U_n$ is a neighborhood of $\infty$. Then for every $0<s<1$ there exist subsequences $\{\vtil_{k_j}\}$ and $\{z_{k_j}\}$, sequences $\{z^\prime_j\}$ and $\{r_j\}$, a contractible periodic Reeb orbit $P'$ for $\lambda$ of period $T'$ such that $|z_{n_j}-z_j^\prime| \rightarrow 0$, $r_j \rightarrow 0^+$, $T' \leq C$ and
\[
 \limsup_{j\rightarrow+\infty} \int_{|z-z_j^\prime|\leq r_j} v_{n_j}^*d\lambda \geq sT'.
\]
\end{lemma}

Let us fix a sequence
\begin{equation}\label{radii_sequence_1}
 R_n \in [0,+\infty] \text{ satisfying } R_n \to +\infty.
\end{equation}
We study a sequence of non-constant $\jtil$-holomorphic maps
\begin{equation}\label{sequence_1}
 \vtil_n = (b_n,v_n) : B_{R_n}(0) \subset \C \to \R \times  S^3 \\
\end{equation}
satisfying
\begin{equation}\label{props_sequence1}
 \begin{aligned}
  & \sup_n E(\vtil_n) = C < +\infty \\
  & \sup_n \int_{B_{R_n}(0) \setminus \D} v_n^*d\lambda \leq \sigma(C) \\
  & \{b_n(2)\} \text{ is bounded}.
 \end{aligned}
\end{equation}
Here $\sigma(C)$ is the constant in~\eqref{sigma}. Define
\begin{equation}\label{set_Gamma}
 \Gamma = \{ z \in \C : \exists n_j \to \infty \text{ and } z_j \to z \text{ satisfying }|d\vtil_{n_j}(z_j)| \to +\infty \}.
\end{equation}
Then, by Lemma~\ref{standardtool},  $\Gamma \subset \D$ and we may assume $\#\Gamma<\infty$ after the selection of a subsequence. We have uniform bounds for the first derivatives of $\vtil_n$ over compact subsets of $\C \setminus \Gamma$. It also follows from standard arguments that after extracting a subsequence of $\vtil_n$, still denoted $\vtil_n$, we can assume the existence of a $\jtil$-holomorphic map $\vtil : \C \setminus \Gamma \to \R \times S^3$ such that
\begin{equation}\label{limit_sequence_1}
 \vtil_n \to \vtil \text{ in } C^\infty_{loc}(\C\setminus\Gamma,\R\times S^3).
\end{equation}
It is not hard to see that if $\Gamma \not= \emptyset$ then $\vtil$ is not constant, but $\vtil$ could be constant if $\Gamma = \emptyset$. We want to rule out this situation and, therefore, assume
\begin{equation}\label{vtil_nonconstant}
  E(\vtil) > 0.
\end{equation}
The inequality $E(\vtil) \leq C$ follows from Fatou's Lemma.

\begin{defi}\label{germ}
A germinating sequence with energy bounded by $C>0$ is a quadruple $(R_n,\vtil_n,\Gamma,\vtil)$ where $R_n$ and $\vtil_n$ are as in (\ref{radii_sequence_1}) and (\ref{sequence_1}) satisfying (\ref{props_sequence1}) and (\ref{limit_sequence_1}) for some $\vtil \in C^\infty(\C \setminus \Gamma,\R \times  S^3)$. Here $\Gamma \subset \C$ is the set defined in (\ref{set_Gamma}), which is assumed finite, and $\vtil$ satisfies (\ref{vtil_nonconstant}).
\end{defi}

\begin{lemma}\label{props_limit_sequence_1}
Fix $C>0$ and let $(R_n,\vtil_n,\Gamma,\vtil)$ be a germinating sequence with energy bounded by $C$. The following assertions are true:
\begin{enumerate}
 \item Every point $z \in \Gamma$ is a negative puncture and $\vtil$ is asymptotic to some periodic Reeb orbit $P_z \in \P$ at $z$.
 \item $\infty$ is a positive puncture and $\vtil$ is asymptotic to some periodic Reeb orbit $P_\infty \in \P$ at $\infty$.
\end{enumerate}
\end{lemma}

The arguments for the proof of the lemma above can be found in the now vast literature concerning bubbling-off of finite-energy surfaces in symplectizations, see~\cite{93,sftcomp,cm,char1,char2,convex,fols}. We only sketch them here.

\begin{proof}[Sketch of the proof of Lemma~\ref{props_limit_sequence_1}]
Assume $z\in\Gamma$ is a positive puncture of $\vtil=(b,v)$. We know from Theorem~\ref{thm93} that there exists a sequence $r_n\to0^+$ such that $\int_{\partial B_{r_n}(z)} v^*\lambda \to -T'$ as $n\to+\infty$, where $T'>0$ is the period of some closed Reeb orbit and $\partial B_{r_n}(z)$ is oriented counter-clockwise. So, fixing $r>0$ small such that $\int_{\partial B_{r}(z)} v^*\lambda < -T'/2$ we find $k$ large satisfying $\int_{B_{r}(z)} v_k^*d\lambda < -T'/2<0$, a contradiction. This proves $z$ is a negative puncture. The assumption $E(\vtil)>0$ now implies that $\infty$ is a non-removable positive puncture. Theorem~\ref{behavior_HWZ_1} finishes the proof.
\end{proof}

As is well-known, a germinating sequence can be seen as the germ of a so-called bubbling-off tree of finite energy spheres. The construction of this tree was first done in the pioneering work~\cite{fols}. We shall not give all the details of the construction and will only consider what we need for our results.

Let $T$ be a finite rooted tree, that is, $T$ is a connected graph with no cycles and with a distinguished vertex called the root. The set of edges will be denoted by $E$ and the set of vertices by $V$. We orient each edge as going away from the root. Thus each vertex distinct of the root has a unique incoming edge coming from its parent and possibly many outgoing edges going to its children. A leaf is a vertex with no children. There is also a level structure on $V$, namely, the level of a vertex is the minimal number of edges necessary to reach the root plus one. Thus, the root is the only vertex in the first level, its children are the vertices in the second level, and so on.

Let us associate a finite energy surface $\util_q : \C \setminus \Gamma \to \R \times  S^3$ to each vertex $q\in V$ such that $\infty$ is the unique positive puncture of $\util_q$ and $\Gamma \subset \C$ is a finite set consisting of negative punctures of $\util_q$. Moreover, the collection $\{\util_q:q\in V\}$ is required to satisfy the following compatibility conditions:
\begin{itemize}
  \item Each edge going out of $q$ corresponds to a unique negative puncture of $\util_q$, and vice-versa.
  \item If $q^\prime$ is a child of $q$ and $z$ is the negative puncture of $\util_q$ corresponding to the edge going from $q$ to $q^\prime$ then there exists a closed Reeb orbit $P$ such that $\util_q$ is asymptotic to $P$ at the negative puncture $z$ and $\util_{q^\prime}$ is asymptotic to $P$ at its positive puncture $\infty$.
\end{itemize}
The set $\{\util_q : q \in V\}$ of finite energy spheres satisfying the above properties and compatibility conditions will be called a \textbf{bubbling-off tree of finite energy spheres modeled on $T$}. The situation considered here is, of course, much simpler than the more general holomorphic buildings treated in~\cite{sftcomp,sft}. For example, according to the above definition, leaves correspond to planes (no outgoing edges = no negative punctures) and the whole tree is topologically a disk.

We shall now indicate how a germinating sequence induces such a bubbling-off tree. This beautiful construction, originally due to Hofer, Wysocki and Zehnder~\cite{fols}, is the precursor of the \textbf{SFT-Compactness Theorem} from~\cite{sftcomp, sft} which adapts \textbf{Gromov's Compactness Theorem} to the more general situations needed in Symplectic Field Theory.

\begin{figure}
  \includegraphics[width=250\unitlength]{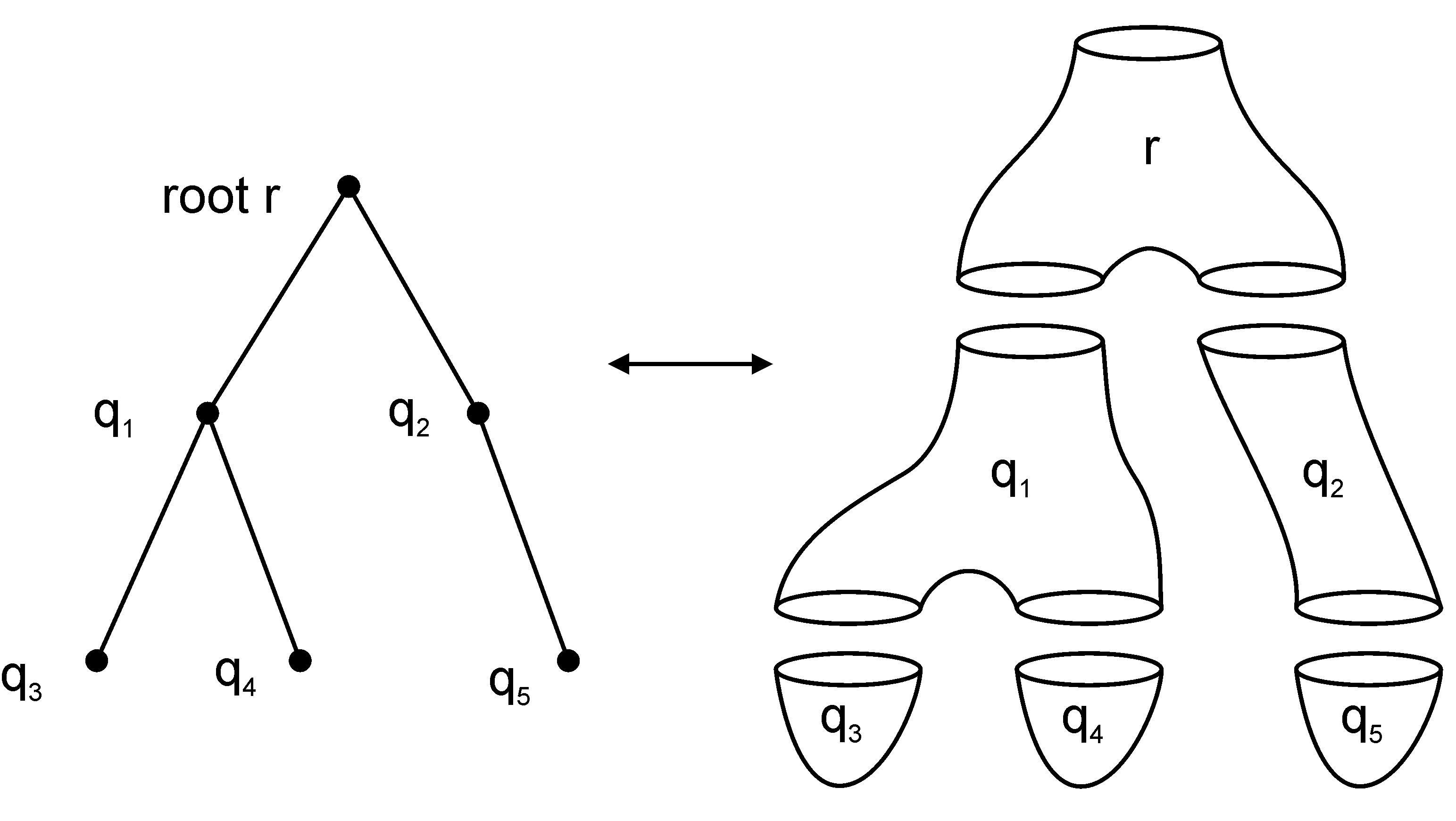}
  \caption{A finite tree $T$ and a bubbling-off tree of finite energy spheres modeled on $T$.}
\end{figure}

Fix $C>0$ and let $(R_n,\vtil_n = (b_n,v_n),\Gamma,\vtil = (b,v))$ be a germinating sequence with energy bounded by $C$, as in Definition~\ref{germ}. Let $\{P_z : z \in \Gamma\}$ and $P_\infty$ be the asymptotic limits at the punctures $\Gamma \cup \{\infty\}$ of $\vtil$. It follows from Lemma~\ref{props_limit_sequence_1} that each $z\in\Gamma$ is a negative puncture and $\infty$ is the unique positive puncture of $\vtil$. Let $T_0$ be the rooted tree consisting of exactly one vertex $\bar q$ (the root) and no edges. We associate the finite energy surface $\vtil$ to the root of $T_0$. If $\Gamma = \emptyset$ then we are done constructing the bubbling-off tree. If not then we proceed.

The mass of a puncture $z\in\Gamma$ is defined in the following manner. For each $\epsilon>0$ small the limit
\[
 m_\epsilon(z) = \lim_{n\to\infty} \int_{B_\epsilon(z)} v_n^*d\lambda = \lim_{n\to\infty} \int_{\partial B_\epsilon(z)} v_n^*\lambda = \int_{\partial B_\epsilon(z)} v^*\lambda
\]
exists in view of (\ref{limit_sequence_1}). It is a non-decreasing function of $\epsilon$ since $\vtil$ is $\jtil$-holomorphic. Following~\cite{fols} one defines
\begin{equation}\label{mass}
  m(z) = \lim_{\epsilon \to 0^+} m_\epsilon(z).
\end{equation}

Fix a negative puncture $z^\prime \in\Gamma$.
It follows easily from Lemma~\ref{standardtool} that
\begin{equation}\label{mass_estimate}
 m(z^\prime) \geq \sigma_1
\end{equation}
where $\sigma_1$ was defined in (\ref{sigma_1}). Fix $\epsilon>0$ small enough so that
\begin{equation}\label{neck_1}
  m_\epsilon(z^\prime) - m(z^\prime) \leq \frac{\sigma(C)}{2}
\end{equation}
where $\sigma(C)$ is given in (\ref{sigma}). Let $z_n \in \cl{B_\epsilon(z^\prime)}$ be such that
\begin{equation}\label{neck_center}
  b_n(z_n) = \inf \{ b_n(\zeta) : \zeta \in \cl{B_\epsilon(z^\prime)} \}.
\end{equation}
Then $z_n \to z^\prime$. This follows easily from the fact that $\vtil_n \to \vtil$ in $C^\infty_{loc}(\C\setminus\Gamma)$ and that $z^\prime$ is a negative puncture of $\vtil$. Since $m_\epsilon(z^\prime) \geq m(z^\prime) \geq \sigma_1 > \sigma(C)$ we can find $0<\delta_n<\epsilon$ such that
\begin{equation}\label{neck_2}
  \int_{B_\epsilon(z^\prime) \setminus B_{\delta_n}(z_n)} v_n^*d\lambda = \sigma(C).
\end{equation}
If $\liminf \delta_n > 0$ choose $0<\epsilon^\prime<\min\{\liminf \delta_n,\epsilon\}$. We get
\[
 \begin{aligned}
  \frac{\sigma(C)}{2} &\geq m_\epsilon(z^\prime) - m(z^\prime) \geq m_\epsilon(z^\prime) - m_{\epsilon^\prime}(z^\prime) \\
  &= \lim_{n\to\infty} \int_{B_\epsilon(z^\prime) \setminus B_{\epsilon^\prime}(z^\prime)} v_n^*d\lambda \\
  &\geq \lim_{n\to\infty} \int_{B_\epsilon(z^\prime) \setminus B_{\delta_n}(z_n)} v_n^*d\lambda = \sigma(C).
 \end{aligned}
\]
This contradiction proves that $\liminf \delta_n = 0$. Thus we can assume $\delta_n \to 0$. Define
\begin{equation}\label{next_level_sequence}
 \begin{aligned}
  & \wtil_n = (d_n,w_n) : B_{R^\prime_n}(0) \to \R \times  S^3 \\
  & \left\{ \begin{aligned} & d_n(z) = b_n(z_n+\delta_nz) - b_n(z_n+2\delta_n) \\ & w_n(z) = v_n(z_n+\delta_nz) \end{aligned} \right.
 \end{aligned}
\end{equation}
where $R^\prime_n \to \infty$ is some sequence such that $B_{\delta_nR^\prime_n}(z_n) \subset B_\epsilon(z^\prime)$. It follows immediately from (\ref{neck_2}) that
\begin{equation}\label{neck_3}
  \limsup \int_{B_{R^\prime_n}(0) \setminus \D} w_n^*d\lambda \leq \sigma(C).
\end{equation}

Now we define
\[
 \Gamma^\prime = \{ z \in \C : \exists n_j\to\infty \text{ and }\zeta_j\to z \text{ such that } |d\wtil_{n_j}(\zeta_j)| \to +\infty \}.
\]
It follows from Lemma~\ref{standardtool} and~\eqref{neck_3} that $\Gamma^\prime\subset\D$ and, up to the choice of a subsequence, we may assume $\#\Gamma^\prime<\infty$. Thus we have uniform bounds of the derivatives of the sequence $\wtil_n$ on compact subsets of $\C\setminus\Gamma^\prime$. Elliptic boot-strapping arguments and the condition $\wtil_n(2) \in \{0\} \times  S^3$ together imply that we can find a smooth $\jtil$-holomorphic map $$ \wtil = (d,w) : \C \setminus \Gamma^\prime \to \R \times S^3 $$ and extract a subsequence of $\wtil_n$, still denoted by $\wtil_n$, such that
\[
 \wtil_n \to \wtil \text{ in }C^\infty_{loc}(\C\setminus\Gamma^\prime,\R\times S^3).
\]
Clearly if $\Gamma^\prime \not= \emptyset$ then $\wtil$ is not constant. This follows easily from Lemma~\ref{standardtool}. If $\Gamma^\prime = \emptyset$ then
\[
 \begin{aligned}
  \int_\D w^*d\lambda &= \lim_{n\to\infty} \int_\D w_n^*d\lambda = \lim_{n\to\infty} \int_{B_{\delta_n}(z_n)} v_n^*d\lambda \\
  &= \lim_{n\to\infty} \int_{B_\epsilon(z^\prime)} v_n^*d\lambda - \int_{B_\epsilon(z^\prime)\setminus B_{\delta_n}(z_n)} v_n^*d\lambda \\
  &= m_\epsilon(z^\prime) - \sigma(C) \geq \sigma_1-\sigma(C) > 0,
 \end{aligned}
\]
which shows that $\wtil$ is not constant in this case as well. Hence
\[
 0 < E(\wtil) \leq C.
\]
We proved that $(R'_n,\wtil_n,\Gamma',\wtil)$ is a germinating sequence with energy bounded by $C$. Lemma~\ref{props_limit_sequence_1} will tell us that $\infty$ is a positive puncture of $\wtil$ and the points of $\Gamma^\prime$ are negative punctures of $\wtil$.

By Theorem~\ref{behavior_HWZ_1} we know $\wtil$ is asymptotic to some closed Reeb orbit $\tilde P$ at the positive puncture $\infty$. Suppose $\vtil$ is asymptotic to $P_{z^\prime}$ at the negative puncture $z^\prime$. One can argue using Theorem~\ref{small_energy} that $\tilde P = P_{z^\prime}$. We reproduce the argument here for the reader's convenience.

Let $\W$ be a $ S^1$-invariant open neighborhood in $C^\infty(S^1,S^3)$ of the set of loops
\[
 \{ t \in  S^1 \mapsto x(Tt) \in  S^3 : P=(x,T) \in \mathcal P(\lambda) \text{ satisfies } T \leq C \}
\]
such that if $P = (x,T) \not= \hat P = (\hat x,\hat T)$ are distinct closed orbits with $\max\{T,\hat T\} \leq C$ then the loops $t\mapsto x(Tt + c)$ and $t\mapsto \hat x (\hat Tt + d)$ belong to distinct components of $\W$, for every $c,d \in \R$. Such $\W$ exists since $\lambda$ is assumed to be non-degenerate. Now, in view of Theorem~\ref{behavior_HWZ_1}, we can find $0<\rho_0<\epsilon$ such that $\rho<\rho_0$ implies that the loop
\begin{equation}\label{loop_up}
 t \mapsto v(z^\prime + \rho e^{i2\pi t})
\end{equation}
belongs to $\W$. By the same token, we find $R_0 \gg 1$ such that if $R^\prime\geq R_0$ then the loop
\begin{equation}\label{loop_down}
 t \mapsto w(R^\prime e^{i2\pi t})
\end{equation}
belongs to $\W$. For any $0<\rho\leq \rho_0$ the loop (\ref{loop_up}) is the limit in $C^\infty( S^1, S^3)$ of the sequence $t \mapsto v_n(z_n + \rho e^{i2\pi t})$. Analogously, for any fixed $R^\prime \geq R_0$ the loop (\ref{loop_down}) is the limit of the sequence $t \mapsto v_n(z_n + \delta_n R^\prime e^{i2\pi t})$. It follows easily from (\ref{mass_estimate}) that
\[
 e := \liminf \int_{\partial B_{\delta_nR_0}(z_n)} v_n^*\lambda > 0.
\]
We can now apply Theorem~\ref{small_energy} to $C$, $e$ and $\W$ to find $h>0$ such that if $\rho \in [\delta_n R_0 e^h,\rho_0e^{-h}]$ then the loop $t\mapsto v_n(z_n + \rho e^{i2\pi t})$ belongs to $\W$, for every $n$. This shows that $t \mapsto v(z^\prime + e^{-h}\rho_0 e^{i2\pi t})$ and $t \mapsto w(e^hR_0 e^{i2\pi t})$ are loops in the same component of $\W$. It follows that $P_{z^\prime} = \tilde P$.

We can now add a new vertex $q^\prime$ to our tree $T_0$, an edge going from the root $\bar q$ to $q^\prime$ and associate the finite energy sphere $\wtil$ to $q^\prime$.


Obviously one can do the same for every negative puncture of the curve $\vtil$ in the set $\Gamma$. After this process we obtain a new tree $T_1$ with root $\bar q$, and we have succeeded in constructing the second level of our final bubbling-off tree.


We now proceed in the same manner with the negative punctures of the curves associated to the second level of tree. If there are any such punctures then we construct the third level, and so on. After each step of the rescaling procedure described above we produce a curve which has non-vanishing $d\lambda$-energy or at least two negative punctures. This observation, first proved in~\cite{errata}, implies that this procedure has to end after a finite number of steps in view of the estimate $\sup_n E(\vtil_n) = C < \infty$.


Thus, the arguments of~\cite{fols} explained above prove

\begin{theo}\label{theorem_tree}
Let $(R_n,\vtil_n=(b_n,v_n),\Gamma,\vtil)$, be a germinating sequence with energy bounded by $C>0$. Then there exists a finite rooted tree $T$, a bubbling-off tree of finite energy spheres $\{\util_q : q \text{ is a vertex of }T \}$ modeled on $T$ and a subsequence $\vtil_{n_j}$ such that the following holds.
\begin{enumerate}
  \item If $\bar q$ is the root then $\util_{\bar q} = \vtil$. 
  \item If $q$ is not the root then we find sequences $\{z_j\} \subset \C$, $\{\delta_j\} \subset\R$ and $\{c_j\} \subset \R$ such that $z_j$ is bounded, $\delta_j \to 0^+$ and $$ \tilde U_j(z) := (b_{n_j}(z_j+\delta_jz)+c_j,v_{n_j}(z_j+\delta_jz)) \to \util_q \text{ in } C^\infty_{loc} \text{ as } j\to\infty. $$
\end{enumerate}
\end{theo}

We take the opportunity to prove a useful lemma. Consider a vertex $q_0$ of $T$ and the associated finite energy sphere
\[
 \util_{q_0} : \C \setminus \Gamma_0 \to \R \times  S^3
\]
in the bubbling-off tree. Here $\Gamma_0 \subset \C$ is the finite set of negative punctures of $\util_{q_0}$ and each element of $\Gamma_0$ corresponds to an edge going out of $q_0$. Suppose that $\util_{q_0}$ is asymptotic to the periodic Reeb orbit $P^+$ at its positive puncture $\infty$ and that $\mu_{CZ}(P^+) \leq 1$. If $\Gamma_0 = \emptyset$ then $\util_{q_0}$ is a finite-energy plane satisfying $\wind_\pi(\util_{q_0}) < 0$, a contradiction. If not we can use Lemma~\ref{wind1} to find a negative puncture $z \in \Gamma_0$ of $\util_{q_0}$ and a closed Reeb orbit $P_1$ such that $\util_{q_0}$ is asymptotic to $P_1$ at $z$ and $\mu_{CZ}(P_1) \leq 1$. The edge of $T$ corresponding to $z$ goes from $q_0$ to one of its children $q_1$. The curve $\util_{q_1}$ is asymptotic to $P_1$ at its (unique) positive puncture. If $q_1$ is not a leaf we can use Lemma~\ref{wind1} to find a negative puncture $z^\prime$ of $\util_{q_1}$ and a closed Reeb orbit $P_2$ such that $\util_{q_1}$ is asymptotic to $P_2$ at $z^\prime$ and $\mu_{CZ}(P_2) \leq 1$. The edge corresponding to $z^\prime$ goes from $q_1$ to one if its children $q_2$. As before, the curve $\util_{q_2}$ is asymptotic to $P_2$ at its (unique) positive puncture. Continuing this process we find a finite path $$ q_0q_1\dots q_N $$ in the tree $T$ such that
\begin{itemize}
  \item $q_i$ is a parent of $q_{i+1}$, $q_N$ is a leaf,
  \item if $i \in \{1,\dots,N\}$ then $\util_{q_i}$ is asymptotic at its positive puncture to a closed Reeb orbit $P_i$ satisfying $\mu_{CZ}(P_i) \leq 1$.
\end{itemize}

Note that $\util_{q_N}$ is a finite energy plane since $q_N$ is a leaf. Its asymptotic limit (at its unique positive puncture) is the orbit $P_N$ satisfying $\mu_{CZ}(P_N) \leq 1$. This implies $\wind_\pi(\util_{q_N}) < 0$. This contradiction proves

\begin{lemma}\label{estimate_cz_1}
Let $(R_n,\vtil_n,\Gamma,\vtil)$ be a germinating sequence with energy bounded by $C>0$. Consider the finite rooted tree $T$, the bubbling-off tree of finite energy spheres $\{\util_q : q \text{ is a vertex of }T \}$ modeled on $T$ and the subsequence $\vtil_{n_j}$ given by applying Theorem~\ref{theorem_tree}. If $q$ is any vertex of $T$ and $\util_q$ is asymptotic to the closed orbit $P$ at its unique positive puncture then $\mu_{CZ}(P) \geq 2$.
\end{lemma}

\subsection{Compactness}

We are ready to state and prove the main result of this section.

\begin{prop}\label{compactness_theorem_1}
Let $(R_n,\vtil_n = (b_n,v_n),\Gamma,\vtil = (b,v))$ be a germinating sequence with energy bounded by $C>0$ such that $\Gamma \not= \emptyset$.
Suppose further that at least one of the following holds.
\begin{enumerate}
  \item[(i)]  $\int_{\C\setminus\Gamma} v^*d\lambda > 0$ and $\wind_\infty(\vtil,\infty) \leq 1$.
  \item [(ii)]  The curve $\vtil$ is asymptotic at $\infty$ to some $P_\infty\in\P$ satisfying $\mu_{CZ}(P_\infty) \leq 2$.
\end{enumerate}
Then there exist $P_0 = (x_0,T_0) \in \P$ and a finite energy plane $\util_0 = (a_0,u_0) : \C \to \R \times  S^3$ satisfying:
\begin{enumerate}
  \item $u_0 : \C \to  S^3$ is an immersion transversal to the Reeb vector.
  \item $\mu_{CZ}(P_0) = 2$ and $\util_0$ is asymptotic to $P_0$ at the puncture $\infty$.
  \item If some $P=(x,T) \in \P$ satisfies $v_n(B_{R_n}(0)) \cap x(\R) = \emptyset \ \forall n$ then $u_0(\C) \cap x(\R) = \emptyset$. If, in addition, $P$ is simply covered and satisfies $\mu_{CZ}(P) \geq 3$ then $\cl{u_0(\C)} \cap x(\R) = \emptyset$.
\end{enumerate}
\end{prop}

The rest of this subsection is dedicated to the proof of the above statement. Let $\vtil_n$ be a sequence as in Proposition~\ref{compactness_theorem_1}. We find a finite rooted tree $T$, a bubbling-off tree of finite energy spheres $$ \{\util_q : q \text{ is a vertex of } T\} $$ modeled on $T$ and a subsequence $\vtil_{n_j}$ such that the conclusions of Theorem~\ref{theorem_tree} are true.

\subsubsection{Estimating the Conley-Zehnder Indices}

By \emph{1.} of Theorem~\ref{theorem_tree} we know that if $\bar q$ is the root of $T$ then $\vtil=\util_{\bar q}$. Let $\Gamma \subset \C$ the set of negative punctures of $\util_{\bar q}$. The curve $\util_{\bar q}$ is asymptotic to a periodic orbit $P_z$ at each $z\in\Gamma$. Lemma~\ref{estimate_cz_1} implies $\mu_{CZ}(P_z) \geq 2 \ \forall z\in\Gamma$. An application of Lemma~\ref{wind2} shows that $\mu_{CZ}(P_z) = 2 \ \forall z\in\Gamma$. The orbits $\{P_z:z\in\Gamma\}$ are precisely the asymptotic limits at the positive punctures of the curves in the second level of the bubbling-off tree. We succeeded in showing that if $q$ is vertex in the second level of $T$ and $\util_q$ is asymptotic to $P$ at its (unique) positive puncture then $\mu_{CZ}(P) = 2$. We can repeat the above argument inductively on each level of the tree, using lemmas~\ref{estimate_cz_1} and~\ref{wind2} at each step, to prove

\begin{lemma}\label{estimate_cz_2}
If a vertex $q$ is not the root and $\util_q$ is asymptotic to the closed orbit $P$ at its (unique) positive puncture then $\mu_{CZ}(P) = 2$.
\end{lemma}

As a consequence of the above lemma we conclude that each edge of $T$ corresponds to a closed Reeb orbit with Conley-Zehnder index equal to $2$.


\subsubsection{Conclusion of the Proof}

Since $\Gamma$ is assumed to be non-empty the tree $T$ obtained by applying Theorem~\ref{theorem_tree} to the sequence $\vtil_n = (b_n,v_n)$ has a leaf $\underline{q}$ distinct of the root $\overline q$. We set $\util_0 := \util_{\underline q}$. By Lemma~\ref{estimate_cz_2} we know that
\begin{equation}\label{}
 \util_0 = (a_0,u_0) : \C \to \R \times  S^3
\end{equation}
is a finite energy plane asymptotic (at $\infty$) to a closed Reeb orbit $P_0$ satisfying
\[
 \mu_{CZ}(P_0) = 2.
\]

We now show that $u_0$ is an immersion transversal to the Reeb vector. Let $\nu^{neg}$ be the largest negative eigenvalue of the asymptotic operator $A_{P_0}$. In view of (\ref{cz_index_alternative}) we have $\mu_{CZ}(P_0) = 2 \Rightarrow \wind(\nu^{neg},\alpha_{P_0}) = 1$. Corollary~\ref{winds_spec} gives $\wind_\infty(\util_0) = \wind_\infty(\util_0,\infty) \leq 1$. Theorem~\ref{winds_thm} shows that $\wind_\infty(\util_0) = 1$ and $\wind_\pi(\util_0) = 0$. As a consequence of the definition of $\wind_\pi$ we conclude that
\begin{equation}\label{u_0_transversal}
 \R R_{u_0(z)} \oplus du_0(z)(T_z\C) = T_{u_0(z)} S^3 \ \forall z\in \C,
\end{equation}
proving $u_0$ is an immersion transversal to the Reeb vector.

Let $P=(x,T) \in \P$ satisfy $v_n(B_{R_n}(0)) \cap x(\R) = \emptyset \ \forall n$. Consider the finite energy immersion
\[
 \begin{array}{ccc}
   F : \C \setminus \{0\} \to \R \times  S^3, &  &  \zeta \mapsto \left( \frac{T \log |\zeta|}{2\pi},x\left(\frac{T \arg \zeta}{2\pi}\right) \right)
 \end{array}
\]
and define
\[
 A = \{ (z,\zeta) \in \C \times (\C\setminus\{0\}) : \util_0(z) = F(\zeta) \}.
\]
Suppose $(z^*,\zeta^*) \in A$ is not an isolated point of $A$. Since both $\util_0$ and $F$ are immersions, it follows from Carleman Similarity Principle that we can find open neighborhoods $\OO$ and $\OO^\prime$ of $z^*$ and $\zeta^*$, respectively, and a holomorphic diffeomorphism $f : \OO \to \OO^\prime$ such that $F \circ f = \util_0$ on $\OO$. This is an immediate consequence of Lemma 2.4.3 of~\cite{mcdsal}. We get a contradiction to (\ref{u_0_transversal}) and prove that $A$ consists only of isolated points if it is non-empty.

Arguing indirectly, suppose $A \not= \emptyset$ and choose $(z^*,\zeta^*) \in A$. By Theorem~\ref{theorem_tree} we find a bounded sequence $\{z_j\} \subset \C$ and real numbers $\delta_j \to 0^+$ and  $c_j$ with the following properties. If
\[
 \tilde U_j(z) := (b_{n_j}(z_j+\delta_jz)+c_j,v_{n_j}(z_j+\delta_jz))
\]
then $\tilde U_j \to \util_0$ in $C^\infty_{loc}$ as $j\to\infty$. In view of (\ref{u_0_transversal}) the maps $\util_0$ and $F$ intersect transversally at the pair $(z^*,\zeta^*)$. By positivity and stability of intersections of pseudo-holomorphic immersions we find $\tilde z_j \to z^*$ such that $\tilde U_j(\tilde z_j) \in F(\C\setminus\{0\})$. Thus the image of the maps $v_{n_j}$ intersect $x(\R)$ if $j$ is large enough, contradicting our hypotheses. We showed $A = \emptyset$ and $u_0(\C) \cap x(\R) = \emptyset$.

If, in addition, $P=(x,T)$ is simply covered and $\mu_{CZ}(P) \geq 3$ then it follows easily from $\mu_{CZ}(P_0) = 2$ and from the description of the Conley-Zehnder index in subsection~\ref{cz_index} that $P$ and $P_0$ are geometrically distinct. The proof of Proposition~\ref{compactness_theorem_1} is now complete.

\section{Compactness of Fast Planes}\label{comp_fast_planes}

We again assume in this section that $\lambda$ is non-degenerate, and choose $J \in \J(\xi,d\lambda)$. Fix $P = (x,T) \in \P$ and $H \subset \R \times S^3$. Define
\begin{equation}\label{}
  \Theta(H,P,\lambda,J) \subset C^\infty(\C,\R\times S^3)
\end{equation}
by requiring that $\util \in \Theta(H,P,\lambda,J)$ if, and only if, $\util = (a,u)$ is a $\jtil$-holomorphic fast finite-energy plane asymptotic to $P$ and satisfying
\begin{equation}\label{norm_theta}
\begin{array}{ccc}
  \util(0) \in H & \text{and} & \int_{\C\setminus\D} u^*d\lambda = \sigma(T).
\end{array}
\end{equation}
The constant $\sigma(C)>0$ is defined in (\ref{sigma}) for any $C>0$. Here $\jtil$ is the almost complex structure on $\R\times S^3$ induced by $J$ via (\ref{jtil}). Define
\begin{equation}\label{}
  \Lambda(H,P,\lambda,J) = \{ \util \in \Theta(H,P,\lambda,J) : \util \text{ is an embedding.}\}
\end{equation}
In the following statement and in the rest of this section we abbreviate $$ \Theta(H,P,\lambda,J) = \Theta(H,P) \text{ and } \Lambda(H,P,\lambda,J) = \Lambda(H,P). $$
By the definition of fast planes, these sets are empty if $P$ is not simply covered.

\begin{theo}\label{comp_fast}
Suppose the orbit $P=(x,T)$ is linked to every $P^* \in \P$ satisfying $\mu_{CZ}(P^*) = 2$. Suppose also $\mu_{CZ}(P) \geq 3$, $H \cap (\R \times x(\R)) = \emptyset$ and $H$ is compact. Then $\Lambda(H,P)$ is compact in $C^\infty(\C,\R\times S^3)$.
\end{theo}

This subsection is devoted to the proof of the above theorem. We shall make use of the following result from~\cite{hryn1}, concerning the perturbation theory of fast planes.

\begin{theo}\label{lemmafredholm}
Let $\lambda$ be any contact form on a closed $3$-manifold $M$ and denote $\xi = \ker \lambda$. Let $J$ be any $d\lambda$-compatible complex structure $J:\xi\rightarrow\xi$. Suppose $\util = (a,u)$ is an embedded fast finite-energy $\jtil$-holomorphic plane asymptotic to a periodic Reeb orbit $P = (x,T)$ at the puncture $\infty$. If $\mu = \mu(\util) \geq 3$ then $u(\C) \cap x(\R) = \emptyset$ and $u : \C \rightarrow M \setminus x(\R)$ is a smooth proper embedding. Moreover, for any $l\geq1$ there exists an open ball $B_r(0)\subset\R^2$ and a $C^l$ embedding $f : \C \times B_r(0) \rightarrow \R\times M$ satisfying the following properties:
\begin{enumerate}
 \item $f(z,0)=\util(z)$.
 \item If $|\tau|<r$ then $f(\cdot,\tau)$ is an embedded fast finite-energy plane in $\R\times M$ asymptotic to $P$ satisfying $\mu(f(\cdot,\tau)) = \mu$.
 \item Fix $\tau_0 \in B_r(0)$ and let $\{\util_n\}$ be a sequence of embedded fast finite-energy planes asymptotic to $P$ satisfying $\util_n \to f(\cdot,\tau_0)$ in $C^\infty_{loc}$ and $\mu(\util_n) = \mu \ \forall n$. Then there exist sequences $\tau_n \rightarrow \tau_0$, $A_n \rightarrow 1$ and $B_n \rightarrow 0$ such that $$ f(A_n z + B_n,\tau_n) = \util_n(z) \ \forall z \in \C $$ if $n$ is large enough.
\end{enumerate}
\end{theo}

\begin{remark}
We point out that results of C. Wendl in~\cite{wendl} could be used as an alternative starting point for proving Theorem~\ref{comp_fast}. However, introducing the necessary notation in order to discuss and apply results from~\cite{wendl} would make our presentation not self-contained.
\end{remark}

The index $\mu(\util)$ above is just the Conley-Zehnder index of $P$ computed with respect to the capping disk given by $\util$. Note that we do not make any non-degeneracy assumptions on $\lambda$. We take the opportunity to state a useful lemma without proof\footnote{The proof of Lemma~\ref{useful_ends} follows from the Banach space set-up defined in~\cite{props3} where the Fredholm theory for embedded finite-energy surfaces is developed. However note that $\lambda$ is possibly degenerate, so it is important that we only deal with planes with non-degenerate asymptotics, see~\cite{hryn1} for details. More precisely, if $\util$ is an embedded fast finite energy plane then we can choose a bundle $N_{\util} \subset \util^*T(\R\times M)$ complementary to $T\util$ which coincides with $\util^*\xi$ near $\infty$. There is a suitable Banach space of sections of $N_{\util}$ with a fixed exponential decay at $\infty$, and the graph of a section models an embedded surface near $\util(\C)$. Such a nearby surface has a $\jtil$-invariant tangent space if, and only if, the corresponding section belongs to the zero set of a suitable Fredholm map. The lemma follows immediately from the exponential decay mentioned before.}.

\begin{lemma}\label{useful_ends}
Let $f=(h,g) : \C\times B_r(0) \to \R\times M$ be the map obtained by Theorem~\ref{lemmafredholm}. If $K \subset B_r(0)$ is compact and $U$ is an open neighborhood of $x(\R)$ in $M$ then there exists $R>0$ such that $g(z,\tau) \in U$ for every $(z,\tau) \in (\C\setminus B_R(0)) \times K$.
\end{lemma}

Fix a sequence
\begin{equation}\label{}
  \{\util_n = (a_n,u_n)\} \in \Lambda(H,P)
\end{equation}
and define
\[
 \Gamma = \{ z \in \C : \exists n_j\to\infty \text{ and } \{z_j\} \subset\C \text{ such that } z_j \to z \text{ and }|d\util_{n_j}(z_j)| \to +\infty \}.
\]
Again we consider norms taken with respect to the metric $g_0$ (\ref{metric_g0}). Clearly $E(\util) = T$ whenever $\util \in \Theta(H,P)$. In view of~\eqref{norm_theta} and of Lemma~\ref{standardtool} we may assume, up to selection of a subsequence, that
\[
\begin{array}{ccc}
  \#\Gamma<\infty & \text{ and }  & \Gamma \subset \D.
\end{array}
\]
Define $\vtil_n = (b_n,v_n) : \C \to \R\times S^3$ by
\begin{equation}\label{}
\begin{array}{cc}
  b_n(z) = a_n(z) - a_n(2), & v_n(z) = u_n(z).
\end{array}
\end{equation}
Since $g_0$ is $\R$-invariant we know $|d\vtil_n|$ is uniformly bounded on compact subsets of $\C \setminus \Gamma$. Standard elliptic boot-strapping arguments show there exists a subsequence, still denoted $\vtil_n$, and a non-constant finite-energy $\jtil$-holomorphic map $\vtil : \C \setminus \Gamma \to \R \times S^3$ such that
\begin{equation}\label{}
  \begin{aligned}
   \vtil(2) &\in \{0\} \times S^3 \\
   \vtil_n \to \vtil \text{ in } &C^\infty_{loc}(\C\setminus \Gamma,\R\times S^3) \\
   E(\vtil) &\leq \sup_n E(\util_n) = T.
  \end{aligned}
\end{equation}
The inequality $E(\vtil)>0$ is easily verified. Since
\[
 \int_{\C\setminus\D} v_n^*d\lambda = \int_{\C\setminus\D} u_n^*d\lambda = \sigma(T) \ \forall n
\]
if we set $R_n = +\infty$ one sees that $(\vtil_n,R_n,\Gamma,\vtil)$ is a germinating sequence with energy bounded by $T$.

Lemma~\ref{props_limit_sequence_1} shows that $\Gamma$ consists of negative punctures of $\vtil$ and $\infty$ is its unique positive puncture. Using Theorem~\ref{small_energy} as in subsection~\ref{bubbling_section} one proves $\vtil$ is asymptotic to $P$ at $\infty$.

\begin{lemma}\label{comp_wind}
$\int v^*d\lambda > 0$, $\Gamma = \emptyset$ and $\wind_\infty(\vtil) = 1$.
\end{lemma}

\begin{proof}
Let us assume, by contradiction, that $\pi \cdot dv \equiv 0$. Consider the trivial cylinder $F : \C\setminus\{0\} \to\R \times S^3$ defined by $$ z=\est \mapsto F(z) = (Ts, x(Tt)). $$ By Theorem~\ref{zero_dlmabda_FES} there exists a non-constant polynomial $p:\C\to\C$ such that $\vtil = F\circ p$. This proves $\Gamma \not= \emptyset$ and $\vtil(\C\setminus\Gamma) \subset \R\times x(\R)$. The degree of $p$ must be $1$ since $P$ is simply covered. If the zero of $p$ lies in $\interior{\D}$ we obtain
\[
 T = \int_{\partial \D} v^*\lambda = \lim_{n\to\infty} \int_{\D} v_n^*d\lambda = T-\sigma(T),
\]
which is impossible and proves that $\Gamma \subset \partial \D$. Consequently, up to selection of a subsequence, $\util_n(0)$ converges to a point in $H \cap (\R\times x(\R))$. This is absurd and shows that $\int_{\C\setminus\D} v^*d\lambda > 0$.

Fix a global non-vanishing section
\[
 Z : S^3 \to \xi = \ker \lambda.
\]
It follows from our assumptions and from Theorem~\ref{behavior_HWZ_1} that $\exists R_0 \gg1$ such that $\pi \cdot dv(z) \not=0$ when $|z|\geq R_0$. Then $\pi \cdot dv_n(z) \not=0$ when $|z|= R_0$ and $n$ is large enough. Moreover,
\[
 l := \wind_\infty(\vtil,\infty) = \wind(t\mapsto \pi \cdot \partial_r v(R_0e^{i2\pi t}), t\mapsto Z(v(R_0e^{i2\pi t}))).
\]
Define
\[
 l_n := \wind(t\mapsto \pi \cdot \partial_r v_n(R_0e^{i2\pi t}), t\mapsto Z(v_n(R_0e^{i2\pi t})))
\]
when $n$ is large. Choose $\rho_n \to +\infty$ so that $\pi\cdot dv_n(z)$ does not vanish if $|z|=\rho_n$ and
\[
 \wind(t\mapsto \pi \cdot \partial_r v_n(\rho_ne^{i2\pi t}), t\mapsto Z(v_n(\rho_ne^{i2\pi t}))) = \wind_\infty(\vtil_n) = \wind_\infty(\util_n) = 1.
\]
Since each $\pi \cdot dv_n$ satisfies an equation of Cauchy-Riemann type, it has only isolated zeros that count positively in the algebraic intersection count with the zero section of $\wedge^{0,1}T^*\C\otimes_J v_n^*\xi$. We can estimate by standard degree theory:
\[
 \begin{aligned}
  0 &\leq \#\{\text{zeros of } \pi\cdot dv_n \text{ on } B_{\rho_n}(0) \setminus B_{R_0}(0)\} \\
  &= \wind(t\mapsto \pi \cdot \partial_r v_n(\rho_ne^{i2\pi t}), t\mapsto Z(v_n(\rho_ne^{i2\pi t}))) \\
  &- \wind(t\mapsto \pi \cdot \partial_r v_n(R_0e^{i2\pi t}), t\mapsto Z(v_n(R_0e^{i2\pi t}))) \\
  &= 1-l_n.
 \end{aligned}
\]
Since $l_n \to l$ this proves
\[
 1-\wind_\infty(\vtil,\infty) \geq 0.
\]
Theorem~\ref{lemmafredholm} implies $v_n(\C) \cap x(\R) = u_n(\C) \cap x(\R) = \emptyset \ \forall n$. Let us assume $\Gamma \not= \emptyset$ and argue indirectly. Applying Proposition~\ref{compactness_theorem_1} to the sequence $\vtil_n$ we obtain a periodic Reeb orbit $P_0$ not linked to $P$ and satisfying $\mu_{CZ}(P_0)=2$. This is a contradiction to the hypotheses of Theorem~\ref{comp_fast}. Consequently we must have $\Gamma = \emptyset$, that is, $\vtil$ is a finite-energy plane satisfying
\[
 0 \leq \wind_\pi(\vtil) = \wind_\infty(\vtil)-1 \leq 0.
\]
\end{proof}

By Lemma~\ref{comp_wind} we have $\Gamma=\emptyset$, $\wind_\infty(\vtil)=1$ and $$ \lim_{n\to+\infty} a_n(2) =:c $$ exists. Thus $\util_n \to \util$ in $C^\infty(\C,\R\times S^3)$ where
\[
 \util(z) = (b(z)+c,v(z)).
\]
Clearly
\[
 \int_\D u^*d\lambda = \lim_{n\to\infty} \int_\D u_n^*d\lambda = T-\sigma(T).
\]
and $\wind_\infty(\util) = \wind_\infty(\vtil) = 1$. In particular, $\util$ is an immersion.

It remains to show that $\util$ is an embedding. Let $\Delta \subset \C \times \C$ be the diagonal and consider
\[
 D := \{ (z_1,z_2) \in \C \times \C \setminus \Delta : \util(z_1) = \util(z_2) \}.
\]
If $D$ has a limit point in $\C \times \C \setminus \Delta$ then we find, using Carleman's Similarity Principle as in~\cite{props2}, a polynomial $p : \C \to \C$ of degree at least $2$ and a $\jtil$-holomorphic map $f : \C \to \R \times S^3$ such that $\util = f\circ p$. This forces zeros of $d\util$. But this is impossible since $\util$ is an immersion, proving that $D$ consists only of isolated points in $\C \times \C \setminus \Delta$. If $D \not= \emptyset$ then, using positivity and stability of self-intersections of pseudo-holomorphic immersions, we obtain self-intersections of the maps $\util_n$. However we know that each $\util_n$ is an embedding by the definition of $\Lambda(H,P)$. This shows $D=\emptyset$ and that $\util$ is an embedding.

\section{Construction of the global sections}\label{non_deg_case}

In this section we again assume $\lambda$ is non-degenerate and prove the sufficiency statement made in Theorem~\ref{main1}. More precisely, we prove the following proposition.

\begin{prop}\label{sufficiency}
Let $\bar P = (\bar x,\bar T)$ be a simply covered and unknotted closed Reeb orbit of a non-degenerate tight contact form $\lambda$ on $ S^3$ such that $\sl(\bar P) = -1$ and $\mu_{CZ}(\bar P) \geq 3$. Assume further that $\bar P$ is linked to every orbit $P\in\P$ satisfying $\mu_{CZ}(P) =2$. Then there exists an open book decomposition with disk-like pages of $ S^3$ adapted to the Reeb dynamics, with binding $\bar x(\R)$.
\end{prop}

\subsection{A special spanning disk for $\bar P$}\label{special_spanning_disk}

Let $F \hookrightarrow  S^3$ be an oriented embedded surface. The contact structure (\ref{contact_str}) induces a singular distribution
\begin{equation}\label{char_dist}
  (\xi \cap TF)^{\bot}
\end{equation}
called the characteristic distribution of $F$. Here $\bot$ means the $d\lambda$-symplectic orthogonal. It is parametrized by a vector field $V$ given by
\begin{equation}\label{vector_V}
 \left\{ \begin{aligned} & i_V\lambda = 0 \\ & i_Vd\lambda = dH - (i_RdH)\lambda \end{aligned} \right.
\end{equation}
where $H$ is any function defined around $F$ having $0$ as a regular value and such that $F \subset H^{-1}(0)$. In other words, $\R V_p = \xi|_p \cap T_pF$ if $V_p\not=0$ and $\xi|_p = T_pF$ if $V_p = 0$.

At a non-degenerate zero $p$ of $V$ the linearization $DV|_p$ is an isomorphism of $\xi|_p = T_pF$. If $a$ and $b$ are the eigenvalues of $DV_p$ then $ab \not= 0$. The point $p$ is called elliptic if $ab > 0$ or hyperbolic if $ab<0$. If $p$ is elliptic and $a$ and $b$ are real numbers then, following~\cite{93}, we call $p$ nicely elliptic.

Denote by $o$ the orientation of $F$ so that the induced orientation on the boundary satisfies $\lambda|_{T\partial F} > 0$. We also have a fiberwise orientation $o^\prime$ of $\xi|_F$ induced by $d\lambda$. A non-degenerate zero $p$ of $V$ is positive if $o$ coincides with $o'$ at $p$, and negative otherwise.

Now we specialize to the case that $F$ is an embedded disk and $\partial F$ is a transverse knot. The following important theorems are proved in~\cite{char1} and~\cite{93}.

\begin{theo}[Hofer, Wysocki and Zehnder]\label{disk_0}
Let $\lambda$ be a tight contact form on $S^3$ and $\xi = \ker \lambda$ be the associated contact structure. Let $L$ be a transverse unknot satisfying $\sl(L) = -1$ and let $F_0$ be an embedded disk spanning $L$. Then there exists an embedded disk $F$ spanning $L$ such that the singular characteristic distribution $\xi \cap TF$ has precisely one positive nicely elliptic singular point $e$. The new disk $F$ can be taken arbitrarily $C^0$-close to $F_0$, and arbitrarily $C^\infty$-close to $F_0$ near the boundary $L$.
\end{theo}

\begin{defi}\label{darboux_chart}
Let $\lambda$ be a contact form on the $3$-manifold $M$. A Darboux chart for $\lambda$ centered at a point $p \in M$ is a pair $(\V,\Psi)$ where $\V$ is an open neighborhood of $p$ in $M$ and $\Psi : \V \to \R^3$ is an embedding satisfying $\Psi(p) = (0,0,0)$ and $\Psi_*\lambda = dz+xdy$ (here $x,y,z$ are Euclidean coordinates on $\R^3$).
\end{defi}

\begin{theo}[Hofer]\label{disk_1}
Let $\lambda$ and $L$ be as in Theorem~\ref{disk_0}, and suppose $F$ is an embedded disk spanning $L$ so that its characteristic foliation has precisely one singular point $e$ which is positive and elliptic. Fix an arbitrary neighborhood $\U$ of $e$ in $S^3$ and let $(\V,\Psi)$ be a Darboux chart centered at $e$. There exists an embedded disk $F'$ spanning $P$ with the following properties.
\begin{enumerate}
  \item $F' \setminus \U = F \setminus \U$ and $e \in F'$.
  \item The point $e$ is the only singularity of the characteristic foliation of $F'$, and is a nicely elliptic singularity.
  \item There exists an open set $G \subset \V \cap \U$ such that $\Psi(F'\cap G) \subset \{z = - \frac{1}{2}xy\}$ and $e\in G$.
\end{enumerate}
\end{theo}

The following proposition is proved in~\cite{hryn1}.

\begin{prop}\label{prop_bnd}
Let $\lambda$ be a contact form on a $3$-manifold $M$ with Reeb vector field $R$, and $P=(x,T)$ be a simply covered unknotted non-degenerate periodic Reeb orbit. Let also $D$ be an embedded disk with $\partial D = x(\R)$ and fix an arbitrary neighborhood $\hat\U$ of $x(\R)$ in $M$. Then there exists an embedded disk $D'$ spanning $P$ with the following properties.
\begin{enumerate}
  \item $D' \setminus \hat \U = D \setminus \hat \U$.
  \item There exists an open neighborhood $\OO \subset D'$ of $\partial D'=P$ such that $\OO \subset \hat \U$ and $R_p \not\in T_pD'$ for every $p \in \OO \setminus \partial D'$.
\end{enumerate}
\end{prop}

Using the above results we prove

\begin{prop}\label{convenientdisk}
Let $\lambda$ be a tight contact form on $S^3$ with associated contact structure $\xi = \ker \lambda$ and Reeb vector field $R$. Let $P=(x,T)$ be a non-degenerate unknotted simply covered periodic Reeb orbit satisfying $\sl(P) = -1$.
Then there exists an embedded disk $\DD_1$ spanning $P$ with the following properties:
\begin{enumerate}
  \item The singular characteristic distribution $\xi \cap T\DD_1$ has precisely one positive nicely elliptic singular point $e$.
  \item There exists a neighborhood $\OO \subset \DD_1$ of $\partial \DD_1$ such that $R_p \not\in T_p\DD_1$ for every $p \in \OO \setminus \partial\DD_1$.
  \item There exists a Darboux chart $(\V,\Psi)$ for $\lambda$ centered at $e$ such that $\Psi(\V\cap\DD_1) \subset \{z=-\frac{1}{2}xy\}$.
  \item If $y$ is a periodic Reeb trajectory and $y(\R) \subset \DD_1$ then $y(\R) = x(\R)$.
\end{enumerate}
\end{prop}

\begin{proof}
Let $D$ be an embedded disk spanning $P$. Applying Proposition~\ref{prop_bnd} to $D$ we can assume there exists a neighborhood $\OO \subset D$ of $\partial D=P$ such that $R_p \not\in T_pD$ for every $p \in \OO \setminus \partial D$. Let $L \subset \OO$ be a transverse (un-)knot $C^\infty$-close to $P$ with $L \cap P = \emptyset$ and $\sl(L) = -1$. Let $F_0 \subset D$ be the disk satisfying $\partial F_0 = L$. We can obviously assume that the embedded strip $S := D \setminus F_0$ is never tangent to $\xi$.

Applying Theorem~\ref{disk_0} we find a smooth embedded disk $F$ spanning $L$ so that the characteristic distribution $TF \cap \xi$ has precisely one positive nicely elliptic point $e$. The disk $$ D' := S \cup F $$ is piecewise $C^1$. Since $L = \partial F_0 = \partial F$ and $F$ is obtained from $F_0$ by a perturbation that can be taken arbitrarily $C^\infty$-small near $L$, we can patch the strip $S$ and the disk $F$ to obtain a smooth disk $\DD_0$. This new disk differs from $D'$ only on an arbitrarily small neighborhood of $L$. Moreover, $e$ is the only singularity of $T\DD_0 \cap \xi$ and it is positive and nicely elliptic. The disk $\DD_0$ satisfies conditions (1) and (2). Now we can apply Theorem~\ref{disk_1} in order to obtain a disk $\DD_1$ which also fulfills condition (3).

In order to obtain condition (4) fix a smooth embedding $f_1:\D\rightarrow M$ such that $f_1(\D)=\DD_1$ and $f_1(0)=e$. For $\delta>0$ define $A_\delta := \overline{B_\delta(0)} \cup (\D\setminus B_{1-\delta}(0))$. If $\delta$ is small enough then (1) and (2) imply
\[
 \R R_{f_1(z)} \cap T_{f_1(z)}\DD_1 = \{0\}, \ \forall z\in A_\delta.
\]
Consider the set
\[
 X = \{ f \in C^\infty(\D,M) : f \equiv f_1 \text{ on } A_\delta \}.
\]
Then $X$ is closed in the complete metric space $C^\infty(\D,M)$ endowed with the $C^\infty$ topology. Hence it is also a complete metric space. For a fixed periodic Reeb trajectory $y:\R \rightarrow M$ we define
\[
 X_y := \{ f \in X : y(\R) \subset f(\D) \}.
\]
By the definition of $\delta$ we know that
\[
 y(\R) \not= x(\R) \text{ and } f \in X_y \Rightarrow y(\R) \subset f(\D\setminus A_\delta).
\]
It is easy to check that $X_y^c$ is open and dense in $X$ if $y(\R) \not = x(\R)$. There are only countably many periodic Reeb trajectories, up to translations in time, since $\lambda$ is non-degenerate. It follows from Baire's category theorem that
\[
 \bigcap \left\{ X_{x^\prime}^c : P^\prime=(x^\prime,T^\prime) \in\P \text{ and } x^\prime(\R) \not= x(\R) \right\}
\]
is residual in $X$. Hence, by an arbitrarily small $C^\infty$-perturbation supported away from $\partial \DD_1 \cup \{e\}$, we may assume that our disk $\DD_1$ contains no periodic Reeb trajectories other than $x(\R)$. Since this perturbation can be taken $C^\infty$-small, it does not create new singular points.
\end{proof}

The above proposition applied to the orbit $\bar P = (\bar x,\bar T)$ of Proposition~\ref{sufficiency} gives us an embedded disk $\DD_1$ spanning $\bar x(\R)$ with special properties. Using equations (\ref{vector_V}) we obtain a vector field $V$ on $\DD_1$ parametrizing (the $d\lambda$-symplectic orthogonal of) the characteristic distribution $\xi \cap T\DD_1$. It is, of course, implicit that $\DD_1$ is oriented so that the Reeb vector is positive on $\partial \DD_1 = \bar x(\R)$. Since we may change $H$ by $-H$ at will in (\ref{vector_V}) there is no loss of generality to assume that $V$ points outward at the boundary. Consequently, $V$ has precisely one zero $e$ in the interior of $\DD_1$ which is a positive nicely elliptic singularity and also a source for the dynamics of $V$.

\subsection{The Bishop Family}\label{the_bishop_family}

Let $\DD_1$ be the disk obtained in the previous subsection by applying Proposition~\ref{convenientdisk} to the orbit $\bar P$ from Proposition~\ref{sufficiency}. Following~\cite{93,char1,char2} we consider the following boundary value problem:

\begin{equation}\label{bishop_disk}
 \left\{
  \begin{aligned}
   & \util = (a,u) : \D \to \R \times  S^3 \text{ is an embedding} \\
   & d\util \cdot i = \jtil(\util) \cdot d\util \\
   & a \equiv 0 \text{ on } \partial \D, \ u(\partial \D) \subset \DD_1 \setminus \{e\} \\
   & u(\partial \D) \text{ winds once positively around } e
  \end{aligned}
 \right.
\end{equation}

Here $i$ denotes the standard complex structure on $\C$. For the last condition to be precise we need to orient the (embedded) loop $u(\partial \D)$ and the disk $\DD_1$. The disk $\DD_1$ is oriented so that $\lambda|_{\partial\DD_1} > 0$. The loop $u(\partial \D)$ is oriented by orienting $\partial \D$ counter-clockwise.

In the Darboux chart $(\V,\Psi)$ given by Proposition~\ref{convenientdisk} note that $\xi = \text{span}\{v_1,v_2\}$ where $v_1 = \partial_x$ and $v_2= \partial_y-x\partial_z$. Choose $J \in \J(\xi,d\lambda)$ so that $Jv_1=v_2$ on $\V$. Then $\jtil$ defined by (\ref{jtil}) is integrable and the disks $\util_\tau = (a_\tau,u_\tau)$ given by
\begin{equation}\label{disks_formulae}
  \begin{aligned}
    a_\tau(s+it) &= \frac{\tau^2}{4} (s^2 + t^2 - 1) \\
    \Psi \circ u_\tau(s+it) &= \left( \tau s , \tau t , -\frac{\tau^2}{2}st \right)
  \end{aligned}
\end{equation}
form a $1$-parameter family of solutions of (\ref{bishop_disk}), for $\tau>0$ small. Note that $\util_\tau \to (0,e)$ in $C^\infty$ as $\tau \to 0^+$. Fixing a disk $\util = (a,u)$ in this family one also notes that \[
 \int_\D u^*d\lambda > 0.
\]

Let $\M$ be the set of solutions $\util = (a,u) : \D \to \R \times  S^3$ of (\ref{bishop_disk}) that in addition satisfy $u(\partial \D) \cap \partial\DD_1 = \emptyset$. It is proved in~\cite{93} that the linearization of $\bar\partial_{\jtil}$ at any $\util \in \M$ is surjective and its Fredholm index equals $4$. There is a $3$-dimensional reparametrization group $G$ of biholomorphisms of $\D$. One can show that $G$ acts smoothly, properly and freely on $\M$. This fact together with an application of the implicit function theorem turns $\M$ into a smooth principal $G$-bundle with a $1$-dimensional base space $\M/G$ without boundary. As explained in~\cite{93} one can use results of D. McDuff from~\cite{dusa} to show that if $\util_n \in \M$ satisfies $\util_n \to \util$ in $C^\infty_{loc}(\D,\R\times S^3)$ and if $\util$ is non-constant then $\util$ solves (\ref{bishop_disk}).

\begin{figure}
  \includegraphics[width=250\unitlength]{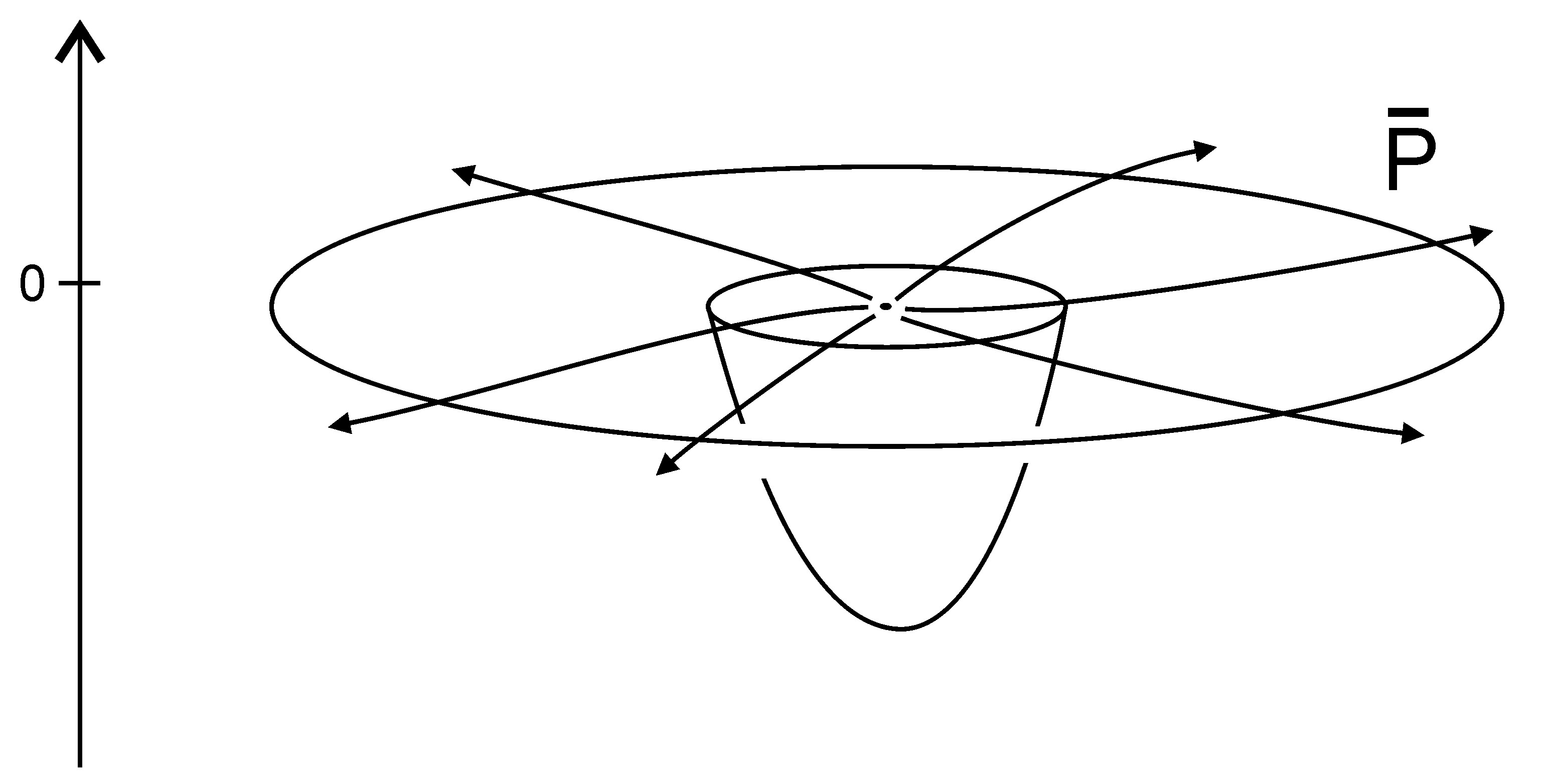}
  \caption{The characteristic foliation and a Bishop disk.}
\end{figure}

There is another very non-trivial fact proved in~\cite{93}: if $\util \in \M$ then every non-vertical vector field in $T_{\util}\M$, seen as a section of $\util^*T(\R\times  S^3)$, is never tangent to the embedded disk $\util(\D)$. We refer the reader to the proof of Theorem 17 from~\cite{93}. There are important consequences. Let $\Pi : \M \to \M/G$ denote the projection. Fix $\util_0 \in \M$ and let $t_0 = \Pi(\util_0)$. If $s$ is a section defined around $t_0$ satisfying $s(t_0) = \util_0$ then there exists a neighborhood $U$ of $t_0$ in $\M/G$ such that the map
\begin{equation}\label{local_embedding}
 \begin{aligned}
  \Phi : U \times \D & \to \R \times  S^3 \\
  (t,z) & \mapsto s(t)(z)
 \end{aligned}
\end{equation}
is a smooth embedding onto its image.

Next one needs to parametrize the $1$-dimensional base space $\M/G$. Each leaf $l$ of the characteristic foliation is a trajectory of the vector $V$. Since there are no other singularities other than $e$, the $\alpha$-limit of $l$ is the source $e$ and $l$ hits $\partial \DD_1$ transversally in forward and finite time. We used that $\partial \DD_1$ is a Reeb trajectory and $V$ points outward at $\partial \DD_1$. Moreover, $l$ has finite length since $e$ is nicely elliptic. The strong maximum principle implies that if $\util\in \M$ then $u(\partial \D)$ intersects the leaves transversally. Since $u(\partial \D)$ winds around $e$ once in $\DD_1$ then it hits every leaf exactly once.

Following~\cite{93} we choose a leaf $l_1$ and denote by $\bar \tau$ its length. Consider the $G$-invariant function $\tau : \M \to \R^+$ given by
\begin{equation}\label{}
 \tau(\util = (a,u)) = \text{length of the piece of } l_1 \text{ connecting } u(\partial\D) \text{ to }e.
\end{equation}
It defines a smooth function on $\M/G$ since each $u(\partial \D)$ intersects $l_1$ once and transversally when $\util \in \M$. The existence of local embeddings as in (\ref{local_embedding}) shows that $\tau : \M/G \to \R^+$ is a local diffeomorphism.

\begin{remark}
Actually one can show that $\tau$ induces a diffeomorphism between a component of $\M/G$ and an open interval, but we can avoid making use of this fact.
\end{remark}

Consider a finite non-empty set $\Gamma \subset \interior{\D}$ and the mixed boundary value problem
\begin{equation}\label{mixed}
 \left\{
  \begin{aligned}
   & \util = (a,u) : \D \setminus \Gamma \to \R \times  S^3 \\
   & \bar\partial_{\jtil}(\util) = 0 \text{ and } \util \text{ is an embedding} \\
   & a \equiv 0 \text{ on } \partial \D, \ u(\partial \D) \subset \DD_1 \setminus \{e\} \\
   & u(\partial \D) \text{ winds once positively around } e \\
   & \int_{\D \setminus \Gamma} u^*d\lambda > 0 \text{ and } E(\util) < \infty \\
   & \text{Every } z\in\Gamma \text{ is a negative puncture.}
  \end{aligned}
 \right.
\end{equation}
Note that we do not fix the complex structure on $\D$.

\begin{prop}\label{bad_problem}
There exists a residual set $\J_{gen} \subset \J(\xi,d\lambda)$ such that the following holds. Fix $J \in \J(\xi,d\lambda)$, let $\util$ be a solution of (\ref{mixed}) and suppose $\util$ is asymptotic to a closed Reeb orbit $P_z$ at each negative puncture $z\in\Gamma$. If $\mu_{CZ}(P_z) \geq 2 \ \forall z\in\Gamma$ and there exists at least one $z_0 \in \Gamma$ satisfying $\mu_{CZ}(P_{z_0}) \geq 3$ then $J \not\in \J_{gen}$.
\end{prop}

\begin{proof}
For the moment we fix $k\geq 1$, $\Gamma \subset \interior{\D}$ with $\#\Gamma = k$, and the asymptotic limits $\{P_z\}_{z\in\Gamma}$. At any solution $\util$ of~\eqref{mixed} we can consider the so-called normal Cauchy-Riemann operator $L$. It is the linearization at zero of a non-linear Fredholm map defined on a space of sections of the normal bundle of $S=\util(\D\setminus \Gamma)$ with the appropriate Sobolev regularity and exponential decay at the punctures. The zero-locus of this map consists of sections representing nearby finite-energy surfaces with the same asymptotic data and boundary condition. When $L$ is surjective the moduli space of such surfaces is locally a finite-dimensional manifold and $\ker L$ is the tangent space at $S$. This analysis is done in~\cite{props3}, where the Fredholm index is computed to be
\[
 \text{Ind } L = \#\Gamma - \sum_{z\in\Gamma} \mu_{CZ}(P_z) + 1.
\]

A delicate argument, also from~\cite{props3}, shows that one can achieve transversality for this problem on a residual subset $\J_{gen}(k,\{P_z\}_{z\in\Gamma}) \subset \J(\xi,d\lambda)$, that is, if $J \in \J_{gen}(k,\{P_z\}_{z\in\Gamma})$ and $\util$ solves~\eqref{mixed} then $L$ is surjective. If the $\{P_z\}$ are as in the statement we can estimate
\begin{equation*}\label{}
 \begin{aligned}
  0 &\leq \text{Ind } L = \#\Gamma - \sum_{z\in\Gamma} \mu_{CZ}(P_z) + 1 \\ &= \#\Gamma - \mu_{CZ}(P_{z_0}) - \sum_{z\not=z_0} \mu_{CZ}(P_z) + 1 \\
  &\leq \#\Gamma - 3 - 2(\#\Gamma-1) + 1 = -\#\Gamma - 3 + 2 + 1 = -\#\Gamma.
 \end{aligned}
\end{equation*}
This contradiction shows that there are no solutions when $J \in \J_{gen}(k,\{P_z\}_{z\in\Gamma})$. The actual location of $\Gamma$ is immaterial since we can move it using a diffeomorphism of $\D$, only the number $k$ and the asymptotic and boundary data are relevant.

Now, there exists a constant $C>0$ depending only on $\DD_1$ and on the non-degenerate contact form $\lambda$ so that any solution of (\ref{mixed}), for any choice of $\Gamma$, satisfies $E(\util)\leq C$. It follows that if such $\util$ is asymptotic to $P = (x,T)$ at some $z\in\Gamma$ then $T\leq C$. It also follows that $k\leq C/\epsilon$, where $\epsilon>0$ satisfies $\epsilon<T'$ for every $P'=(x',T') \in \P$. So there is an upper bound on $k$, depending only on $\DD_1$ and $\lambda$, and only finitely many possibilities for the orbits $P_z$ ($z\in\Gamma$). This means that we only have to consider a finite number of Fredholm problems as above. Since a finite intersection of residual subsets in a complete metric space is still residual, we find $\J_{gen}$ as claimed in the statement of the proposition.
\end{proof}

We need also the following very delicate result from~\cite{char1}.

\begin{theo}[Hofer, Wysocki and Zehnder]\label{delicate}
The set of solutions $\util = (a,u)$ of (\ref{bishop_disk}) satisfying $\util(\D) \cap (\R\times \bar x(\R)) = \emptyset$ is closed in $C^\infty_{loc}$.
\end{theo}

Consider $\util_0 = (a_0,u_0) \in \M$ close to $(0,e)$ satisfying
\begin{equation}\label{}
 \util_0(\D) \cap (\R\times \bar x(\R)) = \emptyset \text{ and } \int_\D u_0^*d\lambda > 0.
\end{equation}
We now perturb $J$ to some $J^\prime \in \J_{gen}$ where the set $\J_{gen}$ is given by Proposition~\ref{bad_problem}. By automatic transversality of solutions of (\ref{bishop_disk}) proved in~\cite{93} 
the solution $\util_0$ is perturbed to a $\jtil^\prime$-holomorphic solution $\util_0^\prime$. We now relabel $J^\prime$ and $\util_0^\prime$ by $J$ and $\util_0$. This shows that we could have assumed $J \in \J_{gen}$ from the start.

Theorem~\ref{delicate} shows that $\util(\D) \cap (\R\times \bar x(\R)) = \emptyset$ whenever $\Pi(\util)$ belongs to the component $\Y$ of $\M/G$ containing $\Pi(\util_0)$. Define
\[
 \tau^* := \sup_{\util\in\pi^{-1}(\Y)} \tau(\util).
\]
Clearly $\tau^* \leq \bar \tau$.

Fix two other leaves $l_i$ and $l_{-1}$ of the characteristic foliation of $\DD_1$ distinct to $l_1$, in such a way that $\{l_1,l_i,l_{-1}\}$ is ordered according to the Reeb field along the boundary. A key ingredient in our arguments is the following statement found in~\cite{char1}.

\begin{theo}[Hofer, Wysocki and Zehnder]\label{bubb_control}
There exists $0<\rho<1$ such that for every sequence $\util_n \in \M$ satisfying $\util_n(w) \in l_w$ for $w \in \{1,i,-1\}$ we have
\[
 \sup_n \sup_{\rho<|z|\leq 1} |d\util_n(z)| < \infty.
\]
\end{theo}

We need to analyze the end of the component $\Y$ in the moduli space $\M$.

\begin{prop}\label{disks_behavior}
The equality $\tau^* = \bar \tau$ holds. Moreover, if we take $\tau_n \to \bar \tau^-$ and choose disks $\util_n = (a_n,u_n) \in \Pi^{-1}(\Y)$ satisfying $\tau(\util_n) = \tau_n$ and the normalization conditions $u_n(1) \in l_1$, $u_n(i) \in l_i$ and $u_n(-1) \in l_{-1}$ then the following assertions are true. There exists a subsequence of $\util_n$, still denoted $\util_n$, such that
\[
 \Gamma_0 = \{z\in\D : \exists n_j\to\infty \text{ and } z_j \to z \text{ such that } |d\util_{n_j}(z_j)| \to \infty \}
\]
consists of a single point in $\interior{\D}$. After reparametrizing we can assume that $\Gamma_0 = \{0\}$ and, moreover,
\[
 \util_n \to F_{\bar P} \text{ in } C^\infty_{loc}(\D\setminus\{0\},\R\times  S^3).
\]
Here $F_{\bar P}$ denotes the map $z = \est \mapsto (\bar Ts,\bar x(\bar Tt))$ on $\D \setminus\{0\}$.
\end{prop}

We now turn to the proof of this proposition. One easily checks that $$ \sup_{\util \in \M} E(\util) =: C < \infty. $$ Fix $\tau_n \to \tau^*$ and choose disks $\util_n = (a_n,u_n) \in \Pi^{-1}(\Y)$ satisfying $\tau(\util_n) = \tau_n$ and the normalization conditions $u_n(w) \in l_w$ for $w\in\{1,i,-1\}$. Define $\Gamma_0 \subset \D$ to be the set of points $z^\prime$ for which $\exists \{z_j\} \subset \D$ and $n_j\to\infty$ such that $|d\util_{n_j}(z_j)| \to \infty$ and $z_j \to z^\prime$. Then $\Gamma_0 \subset \interior{\D}$ by Theorem~\ref{bubb_control} and, up to selection of a subsequence, $\#\Gamma_0 < \infty$ by Lemma~\ref{standardtool}. Elliptic boot strapping arguments give us a smooth $\jtil$-holomorphic map $$ \util_0 = (a_0,u_0) : \D \setminus \Gamma_0 \to \R\times S^3 $$ and further subsequence, still denoted $\util_n$, such that $\util_n \to \util_0$ in $C^\infty_{loc}(\D\setminus\Gamma_0)$. Arguing like in Lemma~\ref{props_limit_sequence_1} we conclude that $\Gamma_0$ consists of negative punctures of $\util_0$.

\begin{lemma}\label{props_u_0_1}
$\Gamma_0 \not= \emptyset$.
\end{lemma}

\begin{proof}
Suppose $\Gamma_0 = \emptyset$. As remarked before, results of D. McDuff imply that $\util_0$ solves (\ref{bishop_disk}). Theorem~\ref{delicate} gives $u_0(\partial \D) \cap \bar x(\R) = \emptyset$. Thus $\util_0 \in \M$ and, clearly, $\tau(\util_0) = \tau^*$. We take a local section $s_0$ of $\M$ defined on a neighborhood $U_0$ of $\Pi(\util_0)$ in $\M/G$ and define $\Phi_0 : U_0\times \D \to \R \times S^3$ by $\Phi_0(t,z)=s_0(t)(z)$. As explained before, $\Phi_0$ is a smooth embedding into $\R\times S^3$. Thus we can find elements $\util \in \Pi^{-1}(\Y)$ satisfying $\tau(\util) > \tau^*$, a contradiction.
\end{proof}

\begin{lemma}\label{props_u_0_2}
$\util_0$ is an embedding, $a_0 \equiv 0$ on $\partial \D$ and $u_0(\partial \D) \subset \DD_1 \setminus \{e\}$ winds once and positively around $e$.
\end{lemma}

\begin{proof}
It is only non-trivial to show that $\util_0$ is an embedding. Note that a strong maximum principle still holds for $a_0$ and, consequently, $\util_0$ is an embedding near $\partial\D$. The conclusion follows from results of D. McDuff on positivity of self-intersections of pseudo-holomorphic maps, since self-intersections or critical points of $\util_0$ would imply self-intersections of the nearby $\util_n$ (a critical point is also seen as some kind of self-intersection in this theory).
\end{proof}

Fix any $z_0 \in \Gamma_0$ and let $P_0 = (x_0,T_0)$ be the unique Reeb orbit such that $\util_0$ is asymptotic to $P_0$ at $z_0$. The mass
\[
\begin{array}{cc}
 m(z_0) = \lim_{\epsilon\to 0^+}m_\epsilon(z_0), &  m_\epsilon(z_0) = \lim_{n\to\infty} \int_{B_\epsilon(z_0)} u_n^*d\lambda
\end{array}
\]
is defined exactly as in subsection~\ref{bubbling_section}. Fix $\epsilon>0$ such that $m_\epsilon(z_0)-m(z_0) \leq \sigma(C)/2$. Choose $z_n$ defined by $a_n(z_n) = \inf a_n(B_\epsilon(z_0))$. As before, it follows that $z_n \to z_0$. Defining $\delta_n$ by $$ \int_{B_\epsilon(z_0)\setminus B_{\delta_n}(z_n)} u_n^*d\lambda = \sigma(C) $$ then $\delta_n \to 0$ and we can find $R_n \to +\infty$ such that $B_{R_n\delta_n}(z_n) \subset B_\epsilon(z_0)$. Proceeding as in the ``soft-rescalling'' done in subsection~\ref{bubbling_section}, define
\begin{equation}\label{vtil_n_seq_2}
 \begin{aligned} & \vtil_n = (b_n,v_n) : B_{R_n}(0) \to \R \times S^3 \\ & b_n(z) = a_n(z_n+\delta_n z) - a_n(z_n+2\delta_n) \text{ and } v_n(z) = u_n(z_n+\delta_n z) \end{aligned}
\end{equation}
and
\begin{equation}\label{gamma_2}
  \Gamma_1 = \{ z \in \C : \exists n_j \to \infty \ \text{ and } \ \zeta_j \to z \text{ such that } |d\vtil_{n_j}(\zeta_j)| \to +\infty \}.
\end{equation}
Again, up to the choice of a subsequence, we may assume $\#\Gamma_1 < \infty$ and, using elliptic boot strapping arguments, we may further assume that there exists a smooth $\jtil$-holomorphic map
\begin{equation}\label{}
  \vtil = (b,v) : \C \setminus \Gamma_1 \to \R \times S^3
\end{equation}
such that $\vtil_n \to \vtil$ in $C^\infty_{loc}(\C\setminus\Gamma_1)$. It is easy to estimate $0<E(\vtil)\leq C$. Arguing exactly as in subsection~\ref{bubbling_section}, using Theorem~\ref{behavior_HWZ_1} and Lemma~\ref{props_limit_sequence_1}, we conclude that $\vtil$ has a unique positive puncture at $\infty$ and that $\vtil$ is asymptotic to $P_0$ at $\infty$. This shows that $R_n$, $\vtil_n$, $\Gamma_1$ and $\vtil$ satisfy the requirements of Lemma~\ref{estimate_cz_1}, and we obtain $\mu_{CZ}(P_0) \geq 2$. Since $z_0 \in \Gamma_0$ was arbitrary we proved

\begin{lemma}
If $z \in \Gamma_0$ and $\util_0$ is asymptotic to $P$ at $z$ then $\mu_{CZ}(P) \geq 2$.
\end{lemma}


%
%
%

As a consequence we will obtain the following statement.

\begin{lemma}\label{zera_dlambda_u_0}
$\pi \cdot du_0 \equiv 0$.
\end{lemma}

\begin{proof}
Let us assume $\int_{\D\setminus\Gamma_0} u_0^*d\lambda > 0$, by contradiction. Since $J \in \J_{gen}$, Proposition~\ref{bad_problem} and the above lemma imply that if $z \in \Gamma_0$ and $\util_0$ is asymptotic to $P$ at $z$ then $\mu_{CZ}(P) = 2$. Then the sequence $\vtil_n$ defined in (\ref{vtil_n_seq_2}) satisfies (ii) from Proposition~\ref{compactness_theorem_1}. It also satisfies all other assumptions in the statement of Proposition~\ref{compactness_theorem_1}. Consequently we obtain $P^* = (x^*,T^*) \in \P$ with $\mu_{CZ}(P^*) = 2$ and an immersed disk $\DD^*$ with boundary $x^*(\R)$ satisfying $\DD^* \cap \bar x(\R) = \emptyset$. Here we used that the orbit $\bar P = (\bar x,\bar T)$ from Proposition~\ref{sufficiency} is simply covered, satisfies $\mu_{CZ}(\bar P) \geq 3$ and $$ \util = (a,u) \in \M \Rightarrow u(\D) \cap \bar x(\R) = \emptyset. $$
This shows that $P^*$ is not linked to $\bar P$, contradicting the hypotheses of Proposition~\ref{sufficiency}.
\end{proof}

This has important consequences. It follows from Lemma~\ref{zera_dlambda_u_0} that there exists a Reeb trajectory $\tilde x$ such that $u_0(\D\setminus\Gamma_0) \subset \tilde x(\R)$. We know by Lemma~\ref{props_u_0_2} that $\util_0$ is not constant. Thus $\tilde x$ is periodic and $\tilde x(\R) = u_0(\partial \D) \subset \DD_1$. By the properties of the disk $\DD_1$ explained in the statement of Theorem~\ref{convenientdisk} we have $\tilde x(\R) = \bar x(\R)$. As a consequence we obtain $\tau^* = \bar \tau$. Consider the map $F:\D\setminus\{0\} \to \R \times S^3$ given by $$ z = \est \mapsto (\bar Ts,\bar x(\bar Tt)). $$ One can prove in a standard fashion, using Lemma~\ref{zera_dlambda_u_0} and Carleman's Similarity Principle, that $\exists k \in \Z^+$ and a holomorphic map $\varphi : \D \to \D$ satisfying
\begin{itemize}
  \item $\varphi(\partial \D) = \partial \D$,
  \item $\Gamma_0 = \varphi^{-1}(0)$ and $\util_0 = F \circ \varphi$ and
  \item $k$ is the degree of $\varphi|_{\partial \D} : \partial \D \to \partial \D$.
\end{itemize}
Recall that $\bar T$ is the minimal positive period of $\bar x$. Thus $u_0(\partial \D)$ winds $k$ times around $e$ in $\DD_1$. It follows from Lemma~\ref{props_u_0_2} that $k=1$ and that $\varphi \in G$. The proof of Proposition~\ref{disks_behavior} is now complete.

\subsection{A special fast plane asymptotic to $\bar P$}\label{existence_special_plane_section}

By Proposition~\ref{disks_behavior} we can find a sequence $\util_n = (a_n,u_n) \in \Pi^{-1}(\Y)$ such that
\begin{equation}\label{}
  \util_n \to F_{\bar P} \text{ in } C^\infty_{loc}(\D\setminus\{0\},\R\times S^3).
\end{equation}
Here $F_{\bar P}:\D\setminus\{0\} \to \R \times S^3$ is the map $z = \est \mapsto (\bar Ts,\bar x(\bar Tt))$. Thus $0$ is a bubbling-off point of $\{\util_n\}$. It has a mass $m(0) = \bar T$. We fix, as before, $\epsilon>0$ such that $m_\epsilon(0) - m(0) \leq \sigma(C)/2$, and choose $z_n \in \D$ satisfying $a_n(z_n) = \inf_\D a_n$. We also fix $\delta_n>0$ such that $u_n^*d\lambda$ integrates to $\sigma(C)$ over $B_\epsilon(0) \setminus B_{\delta_n}(0)$. Then $z_n \to 0$ and $\delta_n \to 0$. Now choose $R_n \to \infty$ so that $\delta_n R_n \to 0$ and define $\vtil_n = (b_n,v_n)$ and $\Gamma_1$ as in (\ref{vtil_n_seq_2}) and (\ref{gamma_2}), respectively. Up to selection of a subsequence we can assume $\#\Gamma_1<\infty$ and the existence of a non-constant finite-energy $\jtil$-holomorphic map $\vtil = (b,v)$ as in (\ref{limit_sequence_1}) such that $\vtil_n \to \vtil$ in $C^\infty_{loc}(\C\setminus\Gamma_1)$. Clearly $\vtil$ has a unique positive puncture at $\infty$ and negative punctures at the points of $\Gamma_1$. Using Theorem~\ref{small_energy} as in subsection~\ref{bubbling_section} one proves that $\vtil$ is asymptotic to $\bar P$ at $\infty$. It is not hard to show that $0 \in \Gamma_1$ if $\Gamma_1 \not= \emptyset$, this is so because $b_n(0) = \inf b_n(B_{R_n}(0))$ and points of $\Gamma_1$ are negative punctures.

Suppose $\pi \cdot dv$ vanishes identically. Then, by Theorem~\ref{zero_dlmabda_FES}, there exists a polynomial $p : \C \to \C$ such that $\Gamma_1 = p^{-1}(0)$ and $\vtil = F_{\bar P} \circ p$. The polynomial $p$ must have degree $1$ since $\bar P$ is simply covered. Thus $\exists A\not=0$ such that $p(z)=Az$. Here we used that $0$ is a (the only) root of $p$. We can now estimate
\[
 \begin{aligned}
  \bar T &= \int_{\partial \D} v^*\lambda = \lim_{n\to\infty} \int_\D v_n^*d\lambda \\
  &= \lim_{n\to\infty} \int_{B_{\delta_n}(0)} u_n^*d\lambda = \lim_{n\to\infty} \int_{B_\epsilon(0)} u_n^*d\lambda - \int_{B_\epsilon(0)\setminus B_{\delta_n}(0)} u_n^*d\lambda \\
  &= m_\epsilon(0)-\sigma(C) = m(0) + m_\epsilon(0) - m(0) - \sigma(C) \\
  &\leq \bar T + \sigma(C)/2 - \sigma(C) = \bar T - \sigma(C)/2.
 \end{aligned}
\]
This contradiction proves

\begin{lemma}\label{nonzerodlmabda}
$\int_{\C\setminus\Gamma_1} v^*d\lambda > 0$.
\end{lemma}

From now on we fix a global non-vanishing section
\begin{equation}\label{section_Z}
  Z : S^3 \to \xi.
\end{equation}
If $f$ is any map defined on a domain of $\C$ we denote
\[
  \begin{array}{ccc}
    \partial_\theta f(z) = \left.\frac{d}{d\theta}\right|_{\theta=0}f(e^{i\theta}z) & \text{ and } & \partial_r f(z) = \left.\frac{d}{dr}\right|_{r=1}f(rz).
  \end{array}
\]

\begin{lemma}\label{nonzerodlambda_un}
The sections $\pi \cdot du_n$ have no zeros if $n$ is large enough.
\end{lemma}

\begin{proof}
By Proposition~\ref{disks_behavior} $\exists n_0 \gg 1$ such that $u_n(\partial \D) \subset \OO\setminus \partial \DD_1$ when $n>n_0$. Here $\OO$ is the neighborhood described in Theorem~\ref{convenientdisk}. If $z\in \partial \D$ and $\pi \cdot du_n(z)=0$ then $\partial_\theta u_n(z)$ is a vector on $\R R(u_n(z)) \cap T_{u_n(z)}\DD_1$. However, the properties of $\OO$ imply that this is the zero vector space if $n>n_0$. Thus $\partial_\theta u_n(z) = 0$ if $\pi \cdot du_n(z)=0$, $n>n_0$ and $z\in S^1$. However the strong maximum principle asserts that $\lambda \cdot \partial_\theta u_n(z) = \partial_r a_n(z) > 0$ for every $n$ and $z\in S^1$. This proves that $\pi \cdot du_n(z) \not= 0$ when $n>n_0$ and $z\in S^1$.

We assert that $\pi \cdot \partial_\theta u_n(z)$ and $V(u_n(z))$ are linearly independent vectors in $\xi_{u_n(z)}$ if $n>n_0$ and $|z|=1$. Arguing indirectly, suppose $|z|=1$ and $c_1,c_2 \in \R$ are so that $$ c_1 \pi \cdot \partial_\theta u_n(z) + c_2 V(u_n(z)) = 0. $$ If $c_1 = 0$ then $c_2 V(u_n(z)) = 0$, which implies $c_2 = 0$. Now assume $c_1 \not= 0$. Then
\[
 0 = c_1 \pi \cdot \partial_\theta u_n(z) + c_2 V(u_n(z)) = \pi \cdot \left( c_1\partial_\theta u_n(z) + c_2 V(u_n(z)) \right)
\]
implying $c_1 \partial_\theta u_n(z) + c_2 V(u_n(z)) \in \R R(u_n(z)) \cap T_{u_n(z)}\DD_1$. It follows from the properties of $\OO$ that
\[
 c_1 \partial_\theta u_n(z) + c_2 V(u_n(z)) = 0 \Rightarrow \partial_\theta u_n(z) = -\frac{c_2}{c_1} V(u_n(z))
\]
if $n>n_0$. This proves $\lambda(u_n(z)) \cdot \partial_\theta u_n(z) = 0$, again contradicting the strong maximum principle and proving our assertion. As a consequence we have
\begin{equation}\label{}
  \wind ( t \mapsto \pi \cdot \partial_\theta u_n(e^{i2\pi t}) , t \mapsto V(u_n(e^{i2\pi t})) ) = 0.
\end{equation}
The algebraic count of zeros of $V$ in $\DD_1$ is $+1$ since its only zero is the positive elliptic singularity $e$. Standard degree theory implies
\begin{equation}\label{}
  \wind ( t \mapsto V(\bar x(\bar Tt)) , t \mapsto Z(\bar x(\bar Tt)) ) = 1.
\end{equation}
This proves
\begin{equation}\label{}
  \wind ( t \mapsto V(u_n(e^{i2\pi t})) , t \mapsto Z(u_n(e^{i2\pi t})) ) = 1
\end{equation}
if $n$ is large enough. Now let $x+iy$ be standard complex coordinates in $\D$. Then
\[
 \begin{aligned}
  \pi &\cdot \partial_\theta u_n(e^{i2\pi t}) \\
  &= \pi \cdot du_n(e^{i2\pi t}) \cdot \left( -\sin(2\pi t)\partial_x + \cos(2\pi t)\partial_y \right) \\
  &= -\sin(2\pi t) \pi \cdot \partial_x u_n(e^{i2\pi t}) + \cos(2\pi t) J(u_n(e^{i2\pi t})) \cdot \pi \cdot \partial_x u_n(e^{i2\pi t}).
 \end{aligned}
\]
This implies
\[
 \wind ( t\mapsto \pi \cdot \partial_\theta u_n(e^{i2\pi t}) , t\mapsto \pi \cdot \partial_x u_n(e^{i2\pi t})) = 1.
\]
We can now compute for $n>n_0$:
\[
 \begin{aligned}
  \wind & \left( t\mapsto \pi \cdot \partial_x u_n \left( e^{i2\pi t} \right) , t\mapsto Z\circ u_n \left( e^{i2\pi t} \right) \right) \\
  &= \wind \left( t\mapsto \pi \cdot \partial_x u_n \left( e^{i2\pi t} \right) , t\mapsto \pi \cdot \partial_\theta u_n \left( e^{i2\pi t} \right) \right) \\
  &+ \wind \left( t\mapsto \pi \cdot \partial_\theta u_n \left( e^{i2\pi t} \right) , t\mapsto V \circ u_n \left( e^{i2\pi t} \right) \right) \\
  &+ \wind \left( t\mapsto V \circ u_n \left( e^{i2\pi t} \right) , t\mapsto Z \circ u_n \left( e^{i2\pi t} \right) \right) \\
  &= - 1 + 0 + 1 = 0.
 \end{aligned}
\]
This shows the algebraic count of zeros of $\pi \cdot \partial_x u_n$ is zero. Thus $\pi \cdot \partial_x u_n$ does not vanish in $\D$ if $n>n_0$ because there are only positive isolated zeros.
\end{proof}

\begin{lemma}\label{finalwindinfty}
$\wind_\infty(\vtil,\infty) = 1$.
\end{lemma}

\begin{proof}
After applying a rotation we can assume $t\mapsto v(Re^{i2\pi t})$ converges to $t\mapsto \bar x(\bar Tt)$ in $C^\infty$ as $R\to\infty$. By Lemma~\ref{nonzerodlmabda} and Theorem~\ref{behavior_HWZ_1} there exists $R_0 \gg1$ such that $\pi \cdot dv(z) \not=0$ if $|z|>R_0$. We compute for $R>R_0$:
\[
 \begin{aligned}
  &\wind(t\mapsto \pi\cdot\partial_r v(Re^{i2\pi t}), t\mapsto Z(v(Re^{i2\pi t}))) \\
  &= \lim_{n\to\infty} \wind(t\mapsto \pi\cdot\partial_r v_n(Re^{i2\pi t}), t\mapsto Z(v_n(Re^{i2\pi t}))) \\
  &= \lim_{n\to\infty} \wind\left( t\mapsto \pi\cdot\left.\frac{d}{d\rho}\right|_{\rho=1} \left[u_n(z_n+\rho R\delta_ne^{i2\pi t})\right], t\mapsto Z(u_n(z_n+R\delta_ne^{i2\pi t})) \right)
 \end{aligned}
\]
By the previous lemma we know that $\pi \cdot du_n$ has no zeros on $\D$. We proceed using standard degree theory:
\[
 \begin{aligned}
  &= \lim_{n\to\infty} \wind(t\mapsto \pi\cdot\partial_r u_n(e^{i2\pi t}), t\mapsto Z(u_n(e^{i2\pi t}))) \\
  &= \lim_{n\to\infty} \wind(t\mapsto \pi\cdot\partial_r u_n(e^{i2\pi t}), t\mapsto \pi\cdot\partial_x u_n(e^{i2\pi t})) \\
  &+ \lim_{n\to\infty} \wind(t\mapsto \pi\cdot\partial_x u_n(e^{i2\pi t}), t\mapsto Z(u_n(e^{i2\pi t}))) \\
  &= 1 + 0 = 1.
 \end{aligned}
\]
Here $x+iy$ are standard complex coordinates in $\D$. The proof is complete in view of the definition of $\wind_\infty$, since $R$ can be taken arbitrarily large.
\end{proof}

The map $\vtil$ is the $C^\infty_{loc}(\C\setminus\Gamma_1,\R\times S^3)$-limit of the sequence $\vtil_n$ satisfying all the requirements of Proposition~\ref{compactness_theorem_1}. Assume $\Gamma_1 \not= \emptyset$. In view of lemmas~\ref{nonzerodlmabda} and~\ref{finalwindinfty} we can apply Proposition~\ref{compactness_theorem_1} and obtain a periodic orbit $P_0$ satisfying $\mu_{CZ}(P_0)=2$ not linked to $\bar P$. This contradiction shows that $\Gamma_1 = \emptyset$, that is, $\vtil$ is a finite-energy plane asymptotic to $\bar P$. Furthermore, it is a fast plane since by Lemma~\ref{finalwindinfty} we know $\wind_\infty(\vtil) = \wind_\infty(\vtil,\infty) = 1$. The identity
\[
 \wind_\pi(\vtil) = \wind_\infty(\vtil) - 1 = 0
\]
shows that $v : \C \to S^3$ is an immersion transversal to the Reeb vector. In particular $\vtil$ is an immersion.

We assert that $\vtil$ is an embedding and argue exactly as in section~\ref{comp_fast_planes}. Let $\Delta \subset \C \times \C$ be the diagonal and consider
\[
 D := \{ (z_1,z_2) \in \C \times \C \setminus \Delta : \vtil(z_1) = \vtil(z_2) \}.
\]
If $D$ has a limit point in $\C \times \C \setminus \Delta$ then we find, using Carleman's Similarity Principle as in~\cite{props2}, a polynomial $p : \C \to \C$ of degree at least $2$ and a $\jtil$-holomorphic map $f : \C \to \R \times S^3$ such that $\vtil = f\circ p$. This forces zeros of $d\vtil$, a contradiction. This proves that $D$ consists only of isolated points in $\C \times \C \setminus \Delta$. If $D \not= \emptyset$ then, using positivity and stability of self-intersections of pseudo-holomorphic immersions, we obtain self-intersections of the maps $\vtil_n$. However we know that each $\vtil_n$ is an embedding since so are the $\util_n$. This shows $D=\emptyset$ and that $\vtil$ is an embedding. We proved

\begin{prop}\label{existence_plane}
Assume $\bar P$ satisfies all the assumptions listed in Proposition~\ref{sufficiency}. Then there exists an embedded fast finite-energy plane asymptotic to $\bar P$.
\end{prop}

\subsection{Open book decompositions}\label{ob_decomp}

The following is Theorem 2.5 from~\cite{hryn1}.

\begin{theo}\label{ob2}
Let $\lambda$ be a non-degenerate contact form on the closed $3$-manifold $M$ and let $\xi=\ker\lambda$ be the associated contact structure. Assume $c_1(\xi)$ vanishes. Let $\jtil$ be the almost complex structure on $\R\times M$ defined by (\ref{jtil}) for some $J\in\J(\xi,d\lambda)$. Suppose there exists an embedded fast $\jtil$-holomorphic finite-energy plane $\util_0$ asymptotic to $P=(x,T)$ with $\mu(\util_0) = k \geq 3$. We also suppose that the set of planes $\Lambda(H,P,\lambda,J)$ is $C^\infty_{loc}$-compact for every compact subset $H \subset \R\times M$ satisfying $H \cap (\R\times x(\R)) = \emptyset$. Then for every $l\geq1$ there exists a $C^l$ map $\util = (a,u) : S^1 \times \C \rightarrow \R \times M$ with the following properties.
\begin{enumerate}
 \item $\util(\vartheta,\cdot)$ is an embedded fast finite-energy plane asymptotic to $P$ at the (positive) puncture $\infty$ satisfying $\mu(\util(\vartheta,\cdot)) = k,\ \forall \vartheta \in S^1$.
 \item $u(\vartheta,\C) \cap x(\R) = \emptyset \ \forall \vartheta\in S^1$ and the map $u : S^1 \times \C \rightarrow M \setminus x(\R)$ is an orientation preserving $C^l$-diffeomorphism.
 \item Each $\cl{u(\vartheta,\C)}$ is a smooth global surface of section for the Reeb dynamics. Moreover, the first return map to each $u(\vartheta,\C)$ preserves an area-form of total area equal to $T$.
\end{enumerate}
\end{theo}

For the definition of $\Lambda(H,P,\lambda,J)$ we refer to Section~\ref{comp_fast_planes}. The integer $\mu(\util_0)$ is simply the Conley-Zehnder index of $P$ with respect to the capping disk induced by $u_0(\C)$. We apply the above theorem to the non-degenerate tight contact form $\lambda$ on $S^3$. Let $\bar P = (\bar x, \bar T)$ be a periodic Reeb orbit as in the statement of Proposition~\ref{sufficiency}. Proposition~\ref{existence_plane} gives us an embedded fast finite-energy plane asymptotic to $\bar P$. Theorem~\ref{comp_fast} shows that for every compact $H \subset \R\times S^3$ satisfying $H \cap (\R\times \bar x(\R)) = \emptyset$ the set $\Lambda(H,\bar P)$ is $C^\infty_{loc}(\C,\R\times S^3)$-compact. Then we can apply Theorem~\ref{ob2} and obtain the desired open book decomposition by embedded disks which are global sections for the Reeb dynamics. This completes the proof of Proposition~\ref{sufficiency} and of sufficiency in Theorem~\ref{main1}.

\section{Proof of Necessity in Theorem~\ref{main1}}\label{section_necessity}

Let $P = (x,T)$ be a periodic orbit of the Reeb flow associated to the non-degenerate tight contact form $\lambda$ on $S^3$, where $T>0$ is its minimal period. Assume $P$ is the binding of an open book decomposition with disk-like pages adapted to $\lambda$. Let $F$ be one of its pages, meaning that $F$ is an embedded disk satisfying $\partial F = P$, is transversal to the Reeb vector field $R$ at $\interior F$ and all orbits geometrically distinct to $P$ intersect $\interior F$ infinitely many times, both forward and backward in time. Since $F$ is an embedded disk, $P$ is unknotted.

Let $P' \subset S^3$ be a periodic orbit geometrically distinct to $P$. Its intersection number with $\interior F$ is non-zero, since it must intersect $\interior F$ and all intersections have the same sign by the transversality of $\interior F$ with $R$. This implies that $0\neq [P']\in H_1(S^3 \setminus P, \Z)$ and, therefore, $P'$ is linked to $P$. In particular, all orbits $P'$ satisfying $\mu_{CZ}(P')=2$ are linked to $P$.

We shall prove now that $\sl(P)=-1$. Consider the orientation on $S^3$ induced by $\lambda \wedge d\lambda > 0$. We orient $\partial F$ pointing in the same direction of $R$ and this induces an orientation $o$ on $F$. The orientation $o'$ on $\xi$ is given by $d\lambda|_{\xi}$. Recall the characteristic distribution $D_{\xi}=(TF \cap \xi)^{\bot}$ on $F$ defined in \eqref{char_dist}. Let $V$ be the vector field representing $D_{\xi}$ defined in \eqref{vector_V}, which is transversal to $\partial F$ pointing outside. Since $R$ is transversal to $\interior F$, all singularities have the same sign and they must be positive observing that $ \int_F d\lambda = \int_P \lambda >0$. It is well known that after a $C^{\infty}$ perturbation of $F$ away from a neighborhood of $\partial F$, $V$ becomes Morse-Smale, meaning that there are no saddle connections and all singular points are non-degenerate. This perturbation keeps the transversality of $R$  with $\interior F$. Let $Z_{\xi}$ and $Z_{TF}$ be non-zero sections of $\xi \to F$ and $TF \to F$, respectively. Since $V$ points outside, we have from standard degree theory
\begin{equation}
 \begin{aligned} \wind(V|_{\partial F},Z_{\xi}|_{\partial F}) & = \sum_{V(z)=0} \sign(DV(z):(T_zD,o_z) \to (\xi_z,o'_z))  \\ & =  \sum_{V(z)=0} \sign(DV(z) : (T_zD,o_z) \to (T_zD,o_z) ) \\ & = \wind(V|_{\partial F},Z_{TF} |_{\partial F}) = 1.  \end{aligned}
\end{equation}
Note that $DV(z)$ at a zero $z$ of $V$ does not depend on whether we consider $V$ either as a section of $\xi \to F$ or as a section of $TF \to F$. This explains the second identity above.

Pushing $\partial F=P$ slightly in the direction of $Z_{\xi}$ we find a new closed curve $P_Z$, which we can assume to be transversal to $\interior F$ and to $\xi$, with the orientation induced by $P$. The self-linking number $\sl(P)$ is defined to be the intersection number of $P_Z$ with $\interior F$. It follows that $\sl(P)=\wind(Z_{\xi}|_{\partial F},V|_{\partial F})= -\wind(V|_{\partial F},Z_{\xi}|_{\partial F}) =-1$.

Now we show that $\mu_{CZ}(P)\geq 3$. We can find neighborhoods $\U\subset S^3$ of $P$, $\V\subset S^1\times \R^2$ of $S^1 \times \{0\}$ and a diffeomorphism $\psi:\U \to \V$ which maps $F \cap \U$ onto the set $F_1:=\{y=0,x\geq 0\} \cap \V$. Here $(x,y)$ are the $\R^2$-coordinates and $z$ is the $S^1$-coordinate. Let $P_1:=S^1 \times \{0\}$. Orienting $\V$ by $dx \wedge dy \wedge dz$, we can assume that $\psi$ preserves orientation, $(\psi_* R)|_P=\frac{1}{T} \partial_z$ and $\xi_1|_{P_1}:=(\psi_* \xi)|_{P_1}=\text{span}\{\partial_x,\partial_y \}|_{P_1}$. With these choices, the orientation induced by $d\lambda|_{\interior F_1}$ is the one induced by $dz \wedge dx|_{\interior F_1}$ and, therefore, the Reeb vector field $\hat X:=\psi_* R$ satisfies \begin{equation}\label{eq.pos} \left. \left < \hat X,   \partial_y  \right > \right|_{\interior F_1} > 0. \end{equation} Since $\wind(Z_{\xi}|_{\partial F},V|_{\partial F})=\sl(P)=-1$, we have $\wind(\hat Z:= \psi_* Z_{\xi}|_P,\hat V:=\psi_* V|_P)=-1$, which implies \begin{equation}\label{eq.wind} \wind \left(\hat Z,\partial_x \right)=-1 \end{equation} as sections of $\xi_1|_{P_1}$.

Suppose $\mu_{CZ}(P)\leq 1$. Then, as explained in Section 2.2, we find from \eqref{eq.wind} a constant $c_0>0$ such that for every solution $y(t) = \rho(t)e^{\theta(t)i}$ of the linearized flow over $P_1$ with respect to the trivialization of $\xi_1|_{P_1}$ induced by $\{\partial_x, \partial_y \}|_{P_1}$, we have $\Delta(y,T):= \theta(T) -\theta(0) < -c_0$. This implies that for $t$ large $\Delta(y,t)< -2\pi$ and, therefore, orbits sufficiently near $P_1$ must wind around the $z$-axis by an angle of less than $-2\pi$, in contradiction with \eqref{eq.pos}. If $\mu_{CZ}(P)=2$, then $P$ is hyperbolic and has a $2$-dimensional stable manifold $W^s$. In these local coordinates, the tangent space of $W^s$ over $P_1$ is spanned by $\left\{ \partial_z, w \right\}|_{P_1}$, where $w$ is a smooth section of $\xi_1|_{P_1}$ corresponding to eigenvectors of $D\phi_T$ at every point of $P_1$. Since $\mu_{CZ}(P)=2$ we must have $\wind(w, \hat Z)=1$ and, therefore, \eqref{eq.wind} implies $\wind \left(w,\partial_x \right)=\wind(w,\hat Z) + \wind(\hat Z,\partial_x)=1-1=0$. It follows that an orbit $\gamma(t)$ on $W^s$ does not wind around the $z$-axis as $t \to +\infty$, contradicting \eqref{eq.pos} and the fact that $F$ is a global surface of section.  We conclude that $\mu_{CZ}(P)\geq 3$ and this finishes the proof  of necessity in Theorem \ref{main1}.

\end{document}